\documentclass{aims}
\usepackage{amsmath}
  \usepackage{paralist}
  \usepackage{graphics} 
  \usepackage{epsfig} 
\usepackage{graphicx}  \usepackage{epstopdf}
 \usepackage[colorlinks=true]{hyperref}
\hypersetup{urlcolor=blue, citecolor=red}

\def\currentvolume{X}
 \def\currentissue{X}
  \def\currentyear{200X}
   \def\currentmonth{XX}
    \def\ppages{X--XX}
     \def\DOI{10.3934/xx.xx.xx.xx}

\newcommand{\hl}{\sl}
\newcommand{\hd}{\sl}
\newcommand{\ba}{\begin{array}}
\newcommand{\ea}{\end{array}}

\newcommand{\const}{\mbox{\rm const}}
\newcommand{\LL}{\mbox{\rm L}}
\newcommand{\Cnt}{\mbox{\rm C}}
\newcommand{\id}{\mbox{\rm id}}
\newcommand{\pr}{\mbox{\rm pr}}
\newcommand{\prsm}{\mbox{\rm \scriptsize pr}}
\renewcommand{\d}{\mathrm{d}}
\newcommand{\e}{\mathrm{e}}

\newcommand{\vol}{\mbox{\rm vol}}
\newcommand{\bs}{{\mathcal B}}
\newcommand{\Gr}{{Gr}}

\newcommand{\con}{{\mathcal C}}

\newcommand{\NN}{\mathbb{N}} \newcommand{\ZZ}{\mathbb{Z}}
 \newcommand{\RR}{\mathbb{R}}

\newcommand{\XX}{X}
\newcommand{\ganz}{\overline{\XX}}
\newcommand{\rand}{\partial\XX}
\newcommand{\joinrand}{\partial^2\XX}

\newcommand{\Lim}{L_\Gamma}          
\newcommand{\radlim}{L_\Gamma^{\small{\text{rad}}}}

\newcommand{\at}{\!\cdot\!}
\newcommand{\width}{\omega}
\newcommand{\xo}{{o}}
\newcommand{\gamo}{{\gamma\xo}}

\newcommand{\be}{\begin{eqnarray*}}
\newcommand{\ee}{\end{eqnarray*}}
\newcommand{\ein}{\,\rule[-5pt]{0.4pt}{12pt}\,{}}

\newcommand{\is}{\mbox{Is}}
\newcommand{\Ax}{\mbox{Ax}}

\newcommand{\supp}{\mbox{supp}}

\newcommand{\st}{\mbox{such}\ \mbox{that}\ }
\newcommand{\wrt}{\mbox{with}\ \mbox{respect}\ \mbox{to}\ }

\DeclareMathOperator*{\esssup}{ess\,sup} 
\DeclareMathOperator*{\essinf}{ess\,inf} 

\newtheorem{theorem}{Theorem}[section]
\newtheorem{corollary}{Corollary}
\newtheorem*{main}{Main Theorem}
\newtheorem{lemma}[theorem]{Lemma}
\newtheorem{proposition}{Proposition}

\theoremstyle{definition}
\newtheorem{definition}[theorem]{Definition}
\newtheorem{remark}{Remark}

\newcommand{\ep}{\varepsilon}

\title[Ergodic geometry for rank one manifolds] 
      {Ergodic geometry for non-elementary rank one manifolds}

\author[Gabriele Link and Jean-Claude Picaud]{}

\subjclass{Primary: 20H10, 22D40; Secondary: 20F67.}
 \keywords{Hopf Tsuji Sullivan dichotomy, Patterson-Sullivan measures, rank one manifolds.}

 \email{gabriele.link@kit.edu}
 \email{jean-claude.picaud@univ-tours.fr}
 
\begin{document}
\bibliographystyle{AIMS.bst}
\maketitle

\centerline{\scshape Gabriele Link$^*$}
\medskip
{\footnotesize
 \centerline{Institut f\"ur Algebra und Geometrie}
   \centerline{Karlsruhe Institute of Technology (KIT)}
   \centerline{Englerstr.~2}
   \centerline{ 76 131 Karlsruhe, Germany}
} 

\medskip

\centerline{\scshape Jean-Claude Picaud}
\medskip
{\footnotesize
 \centerline{ LMPT, UMR 6083}
 \centerline{Facult\'e des Sciences et Techniques}
   \centerline{Parc de Grandmont}
   \centerline{37 200 Tours, France}
}

\bigskip

 \centerline{(Communicated by the associate editor name)}

\begin{abstract}
Let $\XX$ be a Hadamard manifold, and $\Gamma\subset\is(\XX)$ a non-elementary discrete subgroup of isometries of $\XX$ which
contains a rank one isometry. We relate the ergodic theory of the geodesic flow of the quotient orbifold $M=\XX/\Gamma$ to the behavior of the Poincar{\'e} series of $\Gamma$.
Precisely, the aim of this  paper is to extend  the so-called theorem of Hopf-Tsuji-Sullivan -- well-known for manifolds of pinched negative curvature -- 
to the framework of  rank one orbifolds. Moreover, we derive some important properties for $\Gamma$-invariant conformal densities supported on the geometric 
limit set of $\Gamma$.  
\end{abstract}

\section{Introduction}
\subsection{Problem, motivations and results}

Let $\XX$ be a complete simply connected Riemannian manifold of non-positive sectional curvature, $\Gamma\subset\is(\XX)$ a non-elementary discrete group and $M:=X/\Gamma$ the quotient orbifold 
(which will be called non-elementary as $\Gamma$ is). The Poincar\'e series of $\Gamma$ is defined by
 $$P(s;x,y):=\sum_{\gamma\in\Gamma} \e^{-sd(x,\gamma y)},\quad  x,y\in\XX, $$
where $d$ is the Riemannian distance on $\XX$. 
The number 
$$\delta(\Gamma):=\limsup_{R\to +\infty}\frac{1}{R}\ln \#\{\gamma\in\Gamma\colon d(x,\gamma y)\leq R\}$$
is called the {\hl critical exponent} of $\Gamma$ and is independent of $x,y\in\XX$.
Clearly $P(s;x,y)$ converges for
   $s>\delta(\Gamma)$ and  diverges for $s<\delta(\Gamma)$. 
If $P(\delta(\Gamma);x,y)$ diverges, the group $\Gamma$ is said to be {\hl divergent}, otherwise $\Gamma$ is said to be {\hl convergent}.

 Using this Poincar\'e series, a remarkable class of measures $(\mu_x)_{x\in X}$ supported on the geometric boundary $\rand$ of $\XX$ -- a so-called $\Gamma$-invariant conformal density -- has been  
constructed in the CAT$(-1)$ setting (see \cite{MR0450547} and  \cite{MR556586} for the original constructions, 
then  \cite{MR1348871}, \cite{MR1293874}, \cite{MR1207579}  for extensions, and \cite{MR2057305} for a clear account and deep applications of this theory). 

If $\Gamma$ is torsion-free, then the quotient $M=\XX/\Gamma$ is a manifold, and 
a measure $m_\Gamma$ (called {\hl Patterson-Sullivan measure} or in the cocompact setting {\hl Bowen-Margulis measure} for the third reason below) on the unit tangent bundle $SM$ can be derived from the conformal density mentioned above. In the rank one framework (the framework we are in) the Bowen-Margulis measure $m_\Gamma$ has been first introduced by G.~Knieper in \cite{MR1652924}.

For several reasons, Patterson-Sullivan measure $m_{\Gamma}$ (or equivalently the $\Gamma$-in\-variant conformal density  which induces $m_{\Gamma}$) is of central importance in the dyna\-mical studies of the geodesic flow $(g_\Gamma^t)_{t\in\RR}$ acting on $SM$. We only want to mention a few here:

First, for  a compact rank one locally symmetric manifold (which has negative sectional curvature) the measure $m_{\Gamma}$ is an alternative description of the Liouville measure on the unit tangent bundle. 

Second, for real hyperbolic quotients $M=\mathbb{H}^n/\Gamma$, harmonic analysis of $M$ is closely related to the study of conformal densities (see \cite{MR556586}, \cite{MR1041575}, \cite{MR2166367}, \cite{MR1293874}).

Third, it turns out that for compact geometric rank one manifolds  $m_{\Gamma}$ is the unique measure of maximal entropy (see \cite{MR1652924} and references therein, as previous results had been obtained before Knieper's one in negative curvature; see also \cite{MR2097356} for an interesting study in the non-compact and pinched negative curvature case).

And finally, from the local decomposition of $m_{\Gamma}$ and the conformality property of the conditionals (and its main consequence, the famous shadow lemma) solutions to counting 
problems associated with the discrete group $\Gamma$ can be derived (see again \cite{MR2057305}).

The questions we are concerned with are of {\it dynamical} nature. Namely, we are interested in criteria
 for (complete) conservativity (i.e. absence of a wandering set modulo null sets  -- we refer to Section~\ref{BMmeasure} for precise definitions of conservativity, dissipativity and Hopf's decomposition theorem) or ergodicity (i.e. all invariant sets either have full measure or measure zero) of the geodesic flow $(g_{\Gamma}^t)_{t\in\RR}$ acting on $SM$ with 
respect to the invariant measure $m_{\Gamma}$. Such criteria are provided in the CAT$(-1)$ setting by
 Hopf-Tsuji-Sullivan's dichotomy (HTS dichotomy for short), the story of which began maybe with Poincar\'e's
 recurrence theorem applied to Riemann surfaces and later in the 1930's by E. Hopf's seminal
 work (see \cite{Hopf} and \cite{MR0284564}). Actually, it was observed a long time ago  that  either 
the geodesic flow on a hyperbolic Riemann surface is completely conservative and ergodic with respect to Liouville measure
 or it is completely dissipative and non-ergodic. Tsuji's work (see \cite{MR0414898}) in particular points out
 that conservativity is equivalent to ergodicity. 
 
The aim of this paper is to prove HTS dichotomy for  {\hl non-elementary rank one} manifolds $M$ and with respect to Patterson-Sullivan measure. Here a manifold $M$ is called {\hl non-elementary} if  $\Gamma\cong \pi_1(M)$ is non-elementary (see the precise definition at the beginning of Section~\ref{geomesti}), and {\hl rank one} if there exists  a closed geodesic $\sigma$ in $M$ along which there is no perpendicular parallel Jacobi field. 

We want to mention here two important points concerning the notion of geometric rank one: First, in the definition of a geometric rank one manifold it is sometimes not required that the geodesic without perpendicular parallel Jacobi field is closed. However, by the  rank rigidity theorem of 
Ballmann \cite{MR819559} and Burns-Spatzier \cite{MR908215} the different definitions coincide when $M$ is of finite volume; the same is true in general when $M$ has dimension $2$. But to our knowledge this coincidence is not clear in general and we need the periodicity assumption in our main theorem. 

Second, the weakest possible assumption when considering Hadamard manifolds $X$ (and their quotients $M$) that are not higher rank symmetric spaces or Riemannian products would be to require the existence of a geodesic in $X$ which does not bound a flat half plane.  We will call such a geodesic {\hl weak rank one} in the sequel.  
Again, from the rank rigidity result cited above (see also \cite{MR743026}) 
it follows  that for Hadamard manifolds $\XX$ admitting a quotient of finite volume the existence of such a weak rank one geodesic implies that  $\XX$ has a {\hl strong rank one} geodesic, i.e. a geodesic without perpendicular parallel Jacobi field. 
By  the comment below Theorem~2 in \cite{MR799256} the same is true when $\XX$ is $2$-dimensional.  
So if $\XX$ admits a quotient of finite volume or if $\dim(X)=2$, then  every non-elementary manifold $M=\XX/\Gamma$ with a weak rank one geodesic already has a strong rank one geodesic. 
However, we have no idea if the same holds when $\XX$  does not admit a finite volume quotient and $\dim(\XX)\ge 3$.
Roughly speaking, the main result we obtain can be stated  as follows:
 \begin{main} Let $M=X/\Gamma$ be a non-elementary  strong rank one manifold. Then
either the group $\Gamma$ is convergent and the geodesic flow is completely dissipative and non-ergodic \wrt  $m_\Gamma$, or $\Gamma$ is divergent and the geodesic flow is 
completely conservative and ergodic \wrt $m_\Gamma$.
\end{main}

It is worth noticing that given $\Gamma$, the question to decide whether or not the quotient manifold has a conservative -- or ergodic -- geodesic flow with respect to Patterson-Sullivan measure {\hl or} Liouville measure is a different one. 
  In view of HTS dichotomy, the work
 \cite{MR1776078}  in variable negative curvature, as well as \cite{MR1676950} for the specific case of abelian coverings of the 3-punctured sphere should be considered (see also \cite{MR1348871}).

 We further remark that the action of the geodesic flow on $SM$ can be read off from the universal covering $X$ as an action of $\Gamma$ on the space of geodesics (defined uniquely up to 
reparametrization by a unit tangent vector). Each geodesic in $X$ has two end points and  the $\Gamma$-action on that space is actually an action on the space of end point pairs of geodesics 
(denoted $\partial^2X$ in what follows) on which the measure class $(\mu_o\otimes\mu_o)_{| \partial^2X}$ is invariant by the action of $\Gamma$. Nevertheless, a statement 
equivalent to the main theorem above 
for the latter dynamical system can be obtained here only for surfaces; the precise statement can be found in Theorem~\ref{HTS2}.

Notice that in order to recover all our results below except for Proposition~\ref{divseries} and Proposition~\ref{conservative},
the weaker assumption that $M$ admits a closed weak rank one geodesic is sufficient.  In order to  state some of our results under this weaker assumption we briefly recall that a  $\delta$-dimensional  $\Gamma$-invariant conformal density $(\mu_x)_{x\in\XX}$ is a $\Gamma$-equivariant continuous map $\mu$ 
of $X$  to the cone 
of positive finite Borel measures on $\rand$ with support in the geometric limit set of $\Gamma$ and prescribed Radon-Nikodym derivatives. An important subset of the geometric limit
set is the radial limit set, which consists of asymptote classes of geodesic rays for which infinitely many orbit points of $\Gamma$ are contained in a fixed tubular neighborhood;
a precise definition is given at the beginning of Section~\ref{propradlimset}. 

Our first result concerning conformal densities is Lemma~\ref{convseries}: If the group $\Gamma$ is convergent, then the radial limit set is a null set with respect to any measure $\mu_x=\mu(x)$ in the class of 
$\mu$.  Moreover, by Proposition~\ref{ergodicity} the $\Gamma$-action on the radial limit set is ergodic \wrt  the measure $\mu_\xo$: 
\begin{theorem}\label{main2}
 Every $\, \Gamma$-invariant Borel subset of the radial limit set is either a null set or has full measure with respect to $\mu_\xo$.   
\end{theorem}
Finally, according to Proposition~\ref{atomicpart}  the atomic part of a  $\delta$-dimensional  $\Gamma$-invariant conformal density $\mu$ stays outside the radial limit set.

\subsection{Structure of the proof of HTS dichotomy and comments}

Our debt to Roblin's synthesis on HTS dichotomy (see Theorem 1.7 in~\cite{MR2057305}) appears clearly along the lines below.
However, there are  specific geometric arguments due to the rank one hypothesis which allows for   
rather complicated geometric properties of $M$, namely the existence of 
immersed totally geodesic flat or higher rank symmetric spaces (see for instance \cite{MR1132759}, \cite{MR908215},\cite{MR953675},\cite{MR994382} and \cite{MR1050413}).
We note that the additional hypothesis that $\Gamma$ is torsion free  is not relevant 
in the following so that our results apply as well to orbifolds.

For clarity, we  give at this point a brief description of the structure of the proof and precise  steps where new arguments appear to be necessary:
\begin{enumerate}
\item[{\bf 1)}] The shadow lemma for conformal densities (see Section~\ref{shadlem}) is one key geometric argument throughout the proof and needs to be established here for 
generalized shadows with two parameters. 
\item[{\bf 2)}]  It is almost tautological that conservativity of the geodesic flow with respect to $m_{\Gamma}$ is equivalent for the radial limit set $\radlim$ 
of $\Gamma$  
to be of full $\mu_o$-measure. Proposition~\ref{ergodicity} says that $\mu_o(\radlim)$ is either zero or one, and it is an expression of one part of the dichotomy. Its proof 
relies on the above mentioned shadow lemma so that there is no difference in essence from Roblin's  proof. 
\item[{\bf 3)}] Proposition~\ref{divseries}, which is the difficult part in the equivalence between the divergence of the group $\Gamma$ and the conservativity of the geodesic flow requires all the rank one machinery,  in particular the crucial  Proposition~\ref{lrcinproduct}. It is the first ingredient in the proof where the existence of a strong rank one closed geodesic is required. 
\item[{\bf 4)}] The last step is Proposition~\ref{conservative} which asserts that conservativity of the geodesic flow implies ergodicity. Here again we  believe that this proposition will not be true if we assume only the existence of a weak rank one closed geodesic in $M$. We note that the proof is partly based on previous geometric 
work by G. Knieper (\cite{MR1465601} and \cite{MR1652924}). 
\item[{\bf 5)}] Dissipativity of $(\partial^2X,\Gamma,(\mu_o\otimes\mu_o)\ein_{\partial^2X})$ when $\mu_o(\radlim)=0$ is obtained by a variation of an argument of Sullivan (see Lemma~\ref{muradlim0}) only in dimension two.
\item[{\bf 6)}] Simultaneous ergodicity of $(\partial^2\XX,\Gamma,(\mu_o\otimes\mu_o)\ein_{\partial^2\XX})$ and  of the geodesic flow is an evidence.
\end{enumerate}

The paper is organized as follows: In Section~\ref{prelim} we fix some notations and recall basic facts about Hadamard manifolds and discrete groups of
isometries which contain a (weak) rank one element.
In Section~\ref{geomesti} we give some important geometric estimates needed for our generalization of the so-called shadow lemma \cite[Theorem 3.6]{MR2290453}. 
This generalization is stated and proved as Theorem~\ref{shadowlemma} in Section~\ref{shadlem}; it gives an idea of the local behavior of conformal densities and is central in the proof of 
most of our results. In Section~\ref{propradlimset} we introduce the radial limit set of a discrete group of isometries and prove several measure theoretic properties of this set, in particular
Theorem~\ref{main2} and Proposition~\ref{atomicpart}.  
Section~\ref{BMmeasure} recalls the construction of the geodesic flow invariant Patterson-Sullivan measure on the unit tangent bundle of the quotient 
-- whose construction appears in \cite{MR1293874} and \cite{MR1652924}) -- and deals with ergodic properties of the corresponding dynamical system. 
The main result is 
Proposition~\ref{divseries} which states that for divergent groups $\Gamma$ -- under the hypothesis that $\XX/\Gamma$ contains a closed strong rank one geodesic -- the radial limit set has positive and hence
full measure with respect to $\mu_\xo$.

 In Section~\ref{HopfArgument} we complete the proof of the main theorem; the crucial step is Proposition~\ref{conservative} which asserts that conservativity of the geodesic  flow implies ergodicity.

Finally, in Section~\ref{surfaces}  we study the  action of $\Gamma$ on the space of geodesics $\joinrand$  of $\XX$ with respect to the  measure $ (\mu_\xo\otimes\mu_\xo)\ein_{ \partial^2X}$.  Unfortunately we cannot establish complete dissipativity for arbitrary convergent groups; 
it is possible that there remains a conservative part of positive measure in the set of end point pairs of geodesics which are not rank one. However, in dimension two we can exclude this possibility
and therefore get complete dissipativity. Summarizing all results previously obtained we formulate Theorem~\ref{HTS2}  for rank one surfaces which is the  exact analogon of Theorem~1.7 in \cite{MR2057305}.

\section*{Acknowledgements:}    The authors would like to thank Marc Peign{\'e} for many useful discussions. They would also like to thank E.~Gutkin  for pointing out reference \cite{MR1676950}  and for fruitful discussions. Finally, most part of  
the second author's work was completed when visiting {\rm ETHZ} in the fall 2010. The latter is particularly greatful to the institution for providing hospitality and nice 
working conditions; special thanks for that reasons are adressed to 
Marc Burger. 

\section{Preliminaries and notations}\label{prelim}

In this section we recall a few basic properties of Hadamard manifolds which possess
a geodesic without flat half plane. Most of the material can be found in \cite{MR656659} and \cite{MR1377265}.

Let $\XX$ be a complete simply connected Riemannian manifold of
non-positive sectional curvature. The geometric boundary $\rand$ of
$\XX$ is the set of equivalence classes of asymptotic geodesic
rays endowed with the cone topology (see e.g. Chapter~II in~\cite{MR823981}). This boundary is homeomorphic
to the unit tangent space of an arbitrary point in $\XX$ and
$\ganz:=\XX\cup \rand$ is homeomorphic to a closed ball.
Moreover, the iso\-metry group $\is(\XX)$ of $\XX$
has a natural action by homeomorphisms on the geometric boundary.

Recall that a basis for the topology on $\ganz$ is provided by the balls in $\XX$  and the set of  {\hl truncated cones}
\begin{equation}\label{trunccone}
{\mathcal C}_{x,\xi}^{\varepsilon} (T):=\{ z\in\ganz\colon d(x,z)>T,\, \angle_x(z,\xi)<\ep\},
\end{equation}
where $x\in\XX$, $\xi\in\rand$, $\ep>0$ and  $T>0$ are arbitrary.

All geodesics are assumed to have unit speed. For $x\in\XX$, $y\in\XX\setminus\{ x\}$ and
$\xi\in\rand$ we denote by $\sigma_{x,y}$ the unique 
geodesic joining the point $x=\sigma_{x,y}(0)$ to $y$, and by $\sigma_{x,\xi}$ the unique  geodesic emanating from $x$  \st the ray $\sigma_{x,\xi}(\RR_+)$ is in the class of $\xi$.

We say that two distinct points $\xi$, $\eta\in\rand$ can be joined by a
geodesic if there exists a geodesic $\sigma$ with end points
$\sigma(-\infty)=\xi$ and $\sigma(\infty)=\eta$, and we denote by $ \joinrand$  the set of pairs $(\xi,\eta)\in\rand\times\rand$ \st $\xi$ and $\eta$ can be joined by a geodesic. 
For $(\xi,\eta)\in\joinrand$ we denote by
\begin{equation}\label{joiningflat}
(\xi\eta):=\{ x\in\XX \colon \sigma_{x,\eta}(-\infty)=\xi\} 
\end{equation} 
the subset of $\XX$ obtained by the union of all geodesics  joining $\xi$ and  $\eta$. It is well-known (see e.g. Remark 1.11.7 in \cite{MR533654}) that $(\xi\eta)$ is closed and convex and hence a totally geodesic submanifold of $\XX$. We will describe these sets more precisely at the end of this section.

If $\XX$ is CAT$(-1)$, then any pair of distinct boundary 
points $(\xi,\eta)$ belongs to $\joinrand$, and $(\xi\eta)$ is the image of a geodesic which is unique up to reparametrization. In general, the set $\joinrand$ is much smaller compared to $\rand\times\rand$ minus the diagonal due to the possible existence of flat subspaces in $\XX$. 

Let $x, y\in \XX$, $\xi\in\rand$ and $\sigma$ a geodesic ray in the
class of $\xi$. We put 
\begin{equation}\label{buseman}
 \bs_{\xi}(x, y)\,:= \lim_{s\to\infty}\big(
d(x,\sigma(s))-d(y,\sigma(s))\big).
\end{equation}
This number exists, is independent of the chosen ray $\sigma$, and the
function
\[ \bs_{\xi}(\cdot , y):
 \XX \to  \RR,\quad 
x \mapsto \bs_{\xi}(x, y)\]
is called the {\hl Busemann function} centered at $\xi$ based at $y$ (see also Chapter~II in~\cite{MR823981}); its level sets are called  {\hl horospheres} centered at $\xi$. The Busemann function satisfies 
\begin{eqnarray}
|\bs_{\xi}(x, y)|&\le &d(x,y), \nonumber\\
\bs_{g\cdot\xi}(g\at x,g\at y) & =& \bs_{\xi}(x, y)\qquad\mbox{and}\nonumber\\
\bs_{\xi}(x, z)&=&\bs_{\xi}(x, y)+\bs_{\xi}(y,z)\label{cocycle}
\end{eqnarray}
for all $x,y,z\in\XX$, $\xi\in\rand$ and $g\in\is(\XX)$.

If $(\xi, \eta)\in\joinrand$, $y\in\XX$, we define as in \cite{MR1341941} the {\hl Gromov product} of $\xi$ and $\eta$ with respect to $y$  by 
\begin{equation}\label{GromovProd}
\Gr_y(\xi,\eta):=\frac12\big(\bs_\xi(y,z)+\bs_\eta(y,z)\big),\quad\text{where}\quad z\in(\xi\eta).
\end{equation}
It follows immediately from the definition that the Gromov product is non-negative and independent of $z\in (\xi\eta)$.
Notice that in the case of surfaces the Gromov product can be extended to a larger set than $\joinrand$ (see Theorem~B in~\cite{MR2255528}), but for our purposes
 it will be sufficient to consider it on $\joinrand$.

We say that a geodesic $\sigma: \RR\to\XX$ bounds 
 a  {\hl flat strip of width $c> 0$} if there exists a convex subset $i([0,c]\times \RR)$ in $\XX$ isometric to $[0, c]\times\RR$ \st $\sigma(t)=i(0,t)$ for all $t\in \RR$.
A geodesic $\sigma:\RR\to\XX$ is called {\hl weak rank one} or simply {\hl rank one} if  $\sigma$ does not bound a flat strip of infinite width. In this case the number
\begin{equation}\label{widthdef}
\width(\sigma):= \sup\{c> 0\colon \sigma\ \mbox{bounds a flat strip of width}\  c\}
\end{equation}
is called the {\hl width} of $\sigma$. If $\sigma$ does not bound a flat strip of positive width, then we set $\width(\sigma)=0$. 
Furthermore,  we will say that a geodesic $\sigma$ is  {\hl strong rank one}  if $\sigma$ does not admit a perpendicular parallel  Jacobi field.  

\begin{remark}\  {\rm It is clear that a strong rank one geodesic cannot bound a flat strip of positive width. However, the example of a surface with negative Gaussian curvature except along a simple closed geodesic  where the curvature  vanishes shows that a geodesic can be weak rank one of width zero,  but still admitting a perpendicular parallel Jacobi field. }
\end{remark}

The following important lemma states that even though we cannot join any two distinct points in the geometric boundary of $\XX$, given a rank one geodesic we can at least join all points in a neighborhood of its end points. More precisely, we have the following result which is a reformulation of  Lemma III.3.1 in \cite{MR1377265} (see also Lemma~2.1 in \cite{MR1652924}):
\begin{lemma}[Ballmann]\label{joinrankone} \ 
Let $\sigma:\RR\to\XX$ be a rank one geodesic and $c>\width(\sigma)$. 
Then there exist open disjoint neighborhoods $U$ of $\sigma(-\infty)$ and $V$ of $\sigma(\infty)$ in $\ganz$ with the following properties: 
If $\xi\in U$ and $\eta \in V$ then there exists a rank one geodesic joining $\xi$ and $\eta$. For any such geodesic $\sigma'$ we have $d(\sigma'(\RR), \sigma(0))< c$ and  $\width(\sigma')\le 2c$.
Furthermore, if $\sigma$ is strong rank one, then every geodesic joining a pair of points in $U$ and $V$ is strong rank one.
\end{lemma}
\begin{remark}\label{zerowidthnotregular}\ {\rm It may happen that a family of geodesics $\sigma_n$ of width $\width(\sigma_n)\searrow 0$ converges to a geodesic $\sigma$ of width zero which 
-- by the last statement in Lemma~\ref{joinrankone} -- is not strong rank one. In particular, if $\sigma$ is a weak rank one geodesic of width zero which is not strong rank one, then the width of every geodesic in a neighborhood of $\sigma$ may be strictly positive. }
\end{remark}

\begin{definition}\ 
An isometry $\gamma\neq\id$ of $\XX$ is called {\hd axial}, if there exists a constant
$l=l(\gamma)>0$ and a geodesic $\sigma$ \st
$\gamma(\sigma(t))=\sigma(t+l)$ for all $t\in\RR$. We call
$l(\gamma)$ the {\hd translation length} of $\gamma$, and $\sigma$
an {\hd axis} of $\gamma$. The boundary point 
$\gamma^+:=\sigma(\infty)$ (which is actually independent of the chosen axis $\sigma$) is called the {\hd attractive fixed
point}, and $\gamma^-:=\sigma(-\infty)$ the {\hd repulsive fixed
point} of $\gamma$. For $\gamma\in\is(\XX)$ axial we further put
\mbox{$\Ax(\gamma):=\{ x\in\XX\colon  d(x,\gamma x)=l(\gamma)\}$}.
\end{definition}
We remark that $\Ax(\gamma)=(\gamma^-\gamma^+)$ consists of  the union of parallel geodesics
translated by $\gamma$, 
and 
$\overline{\Ax(\gamma)}\cap\rand$ is exactly the set of fixed points of
$\gamma$. The following particular kind of axial isometries will play a crucial role in the sequel.

\begin{definition}\label{hypaxiso}\ 
An isometry $h$ of $\XX$ is called {\hd weak rank one} or simply {\hd rank one} (respectively {\hl strong rank one}) if $h$ is axial and possesses
a weak (respectively strong) rank one axis. Its {\hd width} $\width(h)$  is defined by
 $$\width(h):=\sup\{ d(x, \sigma_{y,h^+}(\RR)) \colon x, y\in\Ax(h)\}.$$
\end{definition}
We remark that the 
set of fixed points of a rank one isometry $h$ consists of precisely the two points $h^+$ and $h^-$ and that every axial isometry commuting with $h$
possesses the same set of invariant geodesics as $h$. Furthermore, if $\sigma\subset\Ax(h)$ is an axis of $h$ then $\width(\sigma)\le \width(h)\le 2 \width(\sigma)$. 

The following important lemma describes the north-south dynamics of rank one isometries:
\begin{lemma}\label{dynrankone}(\cite{MR1377265}, Lemma III.3.3)\ 
\ Let $h$ be a rank one isometry. Then
\begin{enumerate}
\item[(a)]  every point $\xi\in\rand\setminus\{h^+\}$ can be joined
to $h^+$ by a geodesic, and all these geodesics are  rank one, 
\item[(b)] given neighborhoods $U$ of $h^-$ and $V$ of $h^+$ in $\ganz$ 
there exists $N\in\NN$ \st\\
 $h^{-n}(\ganz\setminus V)\subset U$ and
$h^{n}(\ganz\setminus U)\subset V$ for all $n\ge N$.
\end{enumerate}
\end{lemma}

For $x\in \XX$ and $r>0$ we denote  
$B_x(r)\subset\XX$ the open ball of radius $r$ centered at $x\in\XX$. Set
$D:=\{(x,x)\colon x\in\XX\}$  and consider the continuous projection
\begin{equation}\label{pr}
\pr: \ganz\times\XX\setminus D\to\rand,\quad (z,x)\mapsto 
\sigma_{x,z}(-\infty).
\end{equation} 
If $B\subset\XX$, $y\in \ganz\setminus B$, we further set
$\pr_y(B):=\{\pr(y,x)\colon x\in B\}$.

In the sequel it will be convenient to deal with the following sets introduced (up to  small details) by  T.~Roblin (\cite{MR2057305}): 
For
$r>0$, $c>0$  and $x,y\in\XX$ 
 we set
\begin{align} 
{\mathcal O}^+_{r,c}(x,y) &:= \{\xi\in\rand \colon \exists\, z\in B_x(r)\ \st
\sigma_{z,\xi}(\RR_+)\cap B_y(c)\neq\emptyset\},\nonumber \\
{\mathcal O}^-_{r,c}(x,y) &:= \{\xi\in\rand\colon \forall\, z\in B_x(r)\
\mbox{we have}\ \sigma_{z,\xi}(\RR_+)\cap B_y(c)\neq\emptyset\},\nonumber\\
{\mathcal L}_{r,c}(x,y) &:=  \{(\xi,\eta)\in\rand\times\rand \colon \exists\, x'\in
B_x(r)\ \exists\, y'\in B_y(c)\ \st \nonumber \\
&\hspace*{4.8cm}\sigma_{x',y'}(-\infty)=\xi,\,
\sigma_{x',y'}(+\infty)=\eta\}.\label{Lrc}
\end{align}
It is clear from the definitions that 
\begin{equation}\label{shadrelation}
{\mathcal O}^-_{r,c}(x,y)\subseteq \pr_x(B_y(c))\subseteq {\mathcal O}^+_{r,c}(x,y).\end{equation}
Moreover, ${\mathcal O}^-_{r,c}(x,y)$ is non-increasing in $r$ and non-decreasing in $c$ while  ${\mathcal L}_{r,c}(x,y) $
is non-decreasing in both $r$ and $c$.

The following properties are almost tautological:\\
If $r,r',c,c'>0$, $x'\in B_x(r')$ and $y'\in B_y(c')$ with $d(x',y')>r+r'+c+c'$, then
\begin{align}
{\mathcal O}_{r+r',c}^-(x',y)&\subseteq  {\mathcal O}_{r,c}^-(x,y)\subseteq {\mathcal O}_{r,c+c'}^-(x,y'),\label{o-inclus}\\
{\mathcal L}_{r,c}(x,y)&\subseteq  {\mathcal L}_{r+r',c+c'}(x',y'),\label{lrcinclus}\\
{\mathcal L}_{r,c}(x,y) &\supseteq \{(\xi,\eta)\in \joinrand\colon  \eta\in {\mathcal O}^-_{r,c}(x,y),\,\xi\in\pr_\eta(B_x(r))\} .\label{princlus} 
\end{align}

Let $S\XX$ denote the unit tangent bundle of $\XX$ and $p:S\XX\to\XX$ the foot point projection. For $B\subset \XX$ we define 
\[S B:=\{ v\in S\XX\colon pv\in B\} \subset S\XX,\]
i.e. $SB$ is the set of vectors in $S\XX$ with foot point in $B$. 
 Let  $(g^t)_{t\in\RR}$  denote the geodesic flow on $S\XX$. Each $v\in S\XX$ determines a unique geodesic $\sigma_v$ via the assignment
\[ \sigma(t):= p(g^t v),\quad t\in\RR;\]
its extremities $v^-:=\sigma_v(-\infty)$ and $v^+:=\sigma_v(-\infty)$ are called the negative and positive end point of $v$. In particular, we can define the end point projection 
\[ \Pi:S\XX\to \joinrand,\quad v\mapsto (v^-,v^+)\]
which obviously is  invariant under the geodesic flow. Moreover, for $(\xi,\eta)\in\joinrand$ the set of preimages of $\Pi$ consists of all vectors $v\in S\XX$  with end points $v^-=\xi$ and $v^+=\eta$; in particular we have
\[ p (\Pi^{-1}(\xi,\eta)) =(\xi\eta).\]
It is well-known  (see for instance Lemma~2.4 in \cite{MR823981}) that the totally geodesic submanifold $(\xi\eta)\subset \XX$ is isometric to the  product $C_{(\xi\eta)}\times\RR$, where  $C_{(\xi\eta)}$ is a  closed convex subset of $\XX$ which we will describe more 
precisely in the sequel. For $t\in\RR$  and $v\in \Pi^{-1}(\xi,\eta)$  we denote
\begin{equation}\label{Cxieta} H^t(v):=\{ x\in (v^-v^+)\colon  \bs_{v^+}(pv,x)=\bs_{v^-}(x,pv)=t\}
\end{equation}
the intersection of the  submanifold $(\xi\eta)$ with the horosphere centered at $\xi=v^+$ (or equivalently at $\eta=v^-$) passing through the foot
point of $g^tv$; by the cocycle identity for the Busemann function we have \[H^t(v)= H^0(g^tv).\] 
Notice that if $(\xi\eta)$ is the image of a unique geodesic up to reparametrization, then 
$H^t(v)$ consists of exactly one point, namely the foot point of $g^tv$.  In general, the geodesics $\sigma_v(\RR)$ and $\sigma_w(\RR)$ determined by two vectors $v,w\in \Pi^{-1}(\xi,\eta)$ can be different. However, there always exists a unique $\tau\in\RR$ \st the sets $H^\tau(v)=H^0(g^\tau v)$ and $H^0(w)$ coincide. In particular, if $v\in \Pi^{-1}(\xi,\eta)$ then we have
\[ (\xi\eta)=\{ H^t(v)\colon  t\in\RR\}\] and hence $C_{(\xi\eta)}$ can be identified with  the set
$ H^0(v)$ (or equally well with the set $H^t(v)$ for some fixed $t\in\RR$). 

For a point $x\in \XX$ and $(\xi,\eta)\in\joinrand$ we further denote 
\begin{equation}\label{projxxieta} v(x;\xi,\eta)\end{equation} the unique element $v\in S\XX$ whose foot point is the orthogonal projection of $x$ to the totally geodesic submanifold $(\xi\eta)$ and such that  $\Pi(v)=(\xi,\eta)$.

\section{Geometric estimates}\label{geomesti}

Let $\XX$ be a Hadamard manifold and $\Gamma\subset\is(\XX)$ a discrete group.  The {\hl geometric limit set} $\Lim$ of $\Gamma$  is defined by
$\Lim:=\overline{\Gamma\cdot x}\cap\rand,$ 
where $x\in\XX$ is an arbitrary point.  A discrete group $\Gamma\subset \is(\XX)$ is said to be {\hl non-elementary} if the cardinality of $\Lim$ is infinite.

From here on we will require that $\Gamma\subset\is(\XX)$ is  a non-elementary discrete group which contains a rank one isometry $h$ of finite width $\width(h)\geq 0$, and we fix a point $\xo\in\Ax(h)$. 
As $\Gamma$ is non-elementary, there exists one element (actually an infinite number) not commuting with $h$ so that conjugation by such an element gives another rank one isometry of the same 
width whose fixed points are disjoint from those of $h$.
 Furthermore it is  well-known (see e.g. \cite[Proposition~2.8]{MR656659}) that the geometric limit set of $\Gamma$ is minimal, i.e. $\Lim=\overline{\Gamma\cdot\xi}$ for any $\xi\in\Lim$. 
This implies (\cite[Lemma~2.8]{MR2290453}) that for any open subset $O\subset\rand$ with $O\cap\Lim\ne \emptyset$ there exists a finite set $\Lambda \subset \Gamma$ depending on $O$ \st 
\begin{equation}\label{coverLimset}
 \Lim\subseteq \bigcup_{\beta\in\Lambda} \beta O.
\end{equation}

Roughly speaking, the following fact asserts that for suitable $r$ and $c$ the sets 
$ {\mathcal L}_{r,c}(\xo,\gamma \xo)$ are big enough for all but a finite number of elements in $\Gamma$. It will be crucial in the proof of Proposition~\ref{divseries}.

\begin{proposition}\label{lrcinproduct}\
Let $h\in\Gamma$ be the rank one isometry from above, $r_0>\width(h)$ and $U,V\subset\ganz$ the  neighborhoods of $h^-$, $h^+$ provided by Lemma~\ref{joinrankone} for $r_0$. Then there exists a finite set
$\Lambda\subset\Gamma$ such that the following holds: 

For any $c>0$ there exists $R\gg 1$ \st if 
$\gamma \in \Gamma$ satisfies $d(\xo,\gamma \xo)>R$, then for some  $\beta\in\Lambda$ we have
$$ {\mathcal L}_{r,c}(\xo,\beta\gamma \xo)\cap \big(U\times V\big)\supseteq (U \cap \rand) \times {\mathcal O}_{r,c}^-(\xo,\beta\gamma\xo)\qquad\mbox{for all }\ r\ge r_0.$$
\end{proposition}

 \begin{proof} Since $V\subset\ganz$ is an open neighborhood of $h^+$, there exists a truncated cone $\con:=\con_{\xo, h^+}^{\ep}(T)$ as in~(\ref{trunccone})  such that $\con\subseteq V$.  
Let $\Lambda\subset\Gamma$ be a finite and symmetric set such that $\displaystyle{\cup_{\beta\in\Lambda}}\beta \con$ is an open neighborhood of $L_{\Gamma}$ in $\ganz$ 
and set $\rho:=\max\{d(\xo,\beta\xo)\colon \beta\in\Lambda\}$.  We observe that for a fixed constant $c'>0$
\[ \inf\{\ep>0 \colon  \exists\, \xi\in\rand\ \exists\, T\gg 1 \quad \st \ B_{\gamma\xo}(c')\subset \con_{\xo, \xi}^{\ep}(T)\} \to 0\]
as $d(\xo,\gamma\xo)\to\infty$. Consequently, given $c>0$ we can find $R\gg 1$ depending on $c$ and the truncated cone $\con$ such that for any $\gamma\in\Gamma$ 
satisfying $d(\xo,\gamma\xo)>R$ there exists $\beta\in\Lambda$ with 
$B_{\gamma\xo}(c+\rho)\subseteq \beta^{-1}\con$ and $
\pr_{\xo}\big(B_{\gamma\xo}(c+\rho)\big)\subseteq \beta^{-1}\con. $
By construction, 
$$\pr_{\beta\xo}\big(B_{\beta\gamma\xo}(c+\rho)\big)=\beta\pr_{\xo}\big(B_{\gamma\xo}(c+\rho)\big)\subseteq \con\subseteq V,$$
and moreover we have for all $r>0$ 
$${\mathcal O}_{r,c}^-(\xo,\beta\gamma\xo)\subseteq \pr_{\xo}\big(B_{\beta\gamma\xo}(c)\big)
\subseteq \pr_{\beta\xo}\big(B_{\beta\gamma\xo}(c+\rho)\big)\subseteq V.$$
Consequently, if $r\ge r_0$ and  $(\xi, \eta)\in (U\cap\rand)\times {\mathcal O}_{r,c}^-(\xo,\beta\gamma\xo)$,
then $\eta\in V$; by choice of $U$ and $V$ the points $\xi$ and $\eta$ can be joined by a rank one geodesic $\sigma$ \st $d(\sigma(0),\xo)=d(\sigma(\RR), \xo)<r$. 
Since $\sigma(0)\in B_\xo(r)$ and $\eta=\sigma(\infty)\in {\mathcal O}_{r,c}^-(\xo,\beta\gamma\xo)$, the ray $\sigma(\RR^+)$ 
has non-trivial intersection with $B_{\beta\gamma\xo}(c)$, hence $(\xi,\eta)\in {\mathcal L}_{r,c}(\xo,\beta\gamma \xo).$
\end{proof}

From here on we fix a second rank one isometry $g\in\Gamma$  with fixed points distinct from those of $h$. We will need the following  result which is a direct consequence of the north-south dynamics property of rank-one isometries and the key  Lemma~\ref{joinrankone}. It roughly states that for given $r$, the shadows ${\mathcal O}_{r,c}^-(y,\xo)$ are uniformly big with respect to $y\in\XX$ for a suitable 
$c$  bigger than the width $\width(h)$ of $h$. It becomes apparent in the proof of this statement why we need to consider shadows depending on two parameters.
\begin{proposition}\label{smallshadowsarelarge}\ 
Given $r>0$ there exist an  open neighborhood $O\subset\rand$ of $h^+$, $M\in\NN$ 
 and $c_0>0$ such that for all $y\in \XX$ with $d(y,\xo)>r+c_0$
$$ h^M O\subseteq {\mathcal O}^-_{r,c_0}(y,\xo)\quad\mbox{or}\qquad\ g^Mh^M O\subseteq {\mathcal O}^-_{r,c_0}(y,\xo).$$
\end{proposition}

\begin{proof} Fix $r>0$. For $c>\width(h)$ and $\gamma\in\{g,h\}$ we choose neighborhoods $U(\gamma)$ and $V(\gamma)$ according to Lemma~\ref{joinrankone}; upon taking smaller 
neighborhoods these can be assumed to be pairwise disjoint and at distance at least $2r$ from each other.  By Lemma~\ref{dynrankone} (b) 
there exists $M\in\NN$ such that for all $y\in\XX$ we have $B_{\beta y}(r)=\beta (B_y(r))\subset U(h)$ with $\beta=h^{-M}$ or $\beta=h^{-M}g^{-M}$. Consequently,
  Lemma~\ref{joinrankone} 
rephrases as 
$${\mathcal O}^-_{r,c}(\beta y, \xo)\supseteq O:=V(h)\cap\rand,$$
and
the conclusion now follows with
\[ c_0:=c+\max\{ d(\xo,h^{-M}\xo),d(\xo,h^{-M}g^{-M}\xo)\}\] from the obvious relation\\[2mm]
$\hspace*{35mm} \displaystyle \beta^{-1} {\mathcal O}^-_{r,c}(\beta y, \xo)= {\mathcal O}^-_{r,c}(y,\beta^{-1}\xo)\stackrel{(\ref{o-inclus})}{\subseteq} {\mathcal O}^-_{r,c_0}(y, \xo).$
\end{proof}

\section{The generalized shadow lemma for conformal densities} \label{shadlem}
Given $\delta\ge 0$, a $\delta$-dimensional  $\Gamma$-invariant  conformal density is a continuous map $\mu$ of $\XX$ into the cone of positive finite Borel
measures on $\rand$ \st $\mu_\xo:=\mu(\xo)$ is supported on the limit set $\Lim$, $\mu$ is $\Gamma$-equivariant (i.e. $\gamma_*\mu_x=\mu_{\gamma x}$ for 
all $\gamma\in\Gamma$, $x\in\XX$)\footnote{Here $\gamma_*\mu_x$ denotes the measure defined by $\gamma_*\mu_x(E)=\mu_x(\gamma^{-1}E)$ for any Borel set $E\subseteq\rand$.}  and 
\begin{equation}\label{conformality}
  \frac{\d \mu_x}{\d \mu_\xo}(\eta)=\e^{\delta
\bs_{\eta}(\xo,x)} \quad\mbox{for any}\ \,x\in\XX\ \text{and }\  \eta\in\supp(\mu_\xo).
\end{equation}

The existence of  a  $\delta$-dimensional   $\Gamma$-invariant conformal density for $\delta=\delta(\Gamma)$ goes back to  S.~J.~Patterson \mbox{(\cite{MR0450547})} in the case of Fuchsian groups, and 
it turns out that his explicit construction extends {\it verbatim} to arbitrary infinite discrete isometry groups of Hadamard manifolds with positive critical exponent  
(see e.g. \cite[Lemma 2.2]{MR1465601}). Notice that in our setting the critical exponent $\delta(\Gamma)$ is always positive since $\Gamma$ contains a non-abelian free group generated by two rank one elements.

One corner stone result concerning these densities is Sullivan's shadow lemma, which gives an asymptotic estimate for the measure of  projected balls 
$\pr_\xo(B_{\gamma\xo}(r))$ as $d(\xo,\gamma\xo)$ tends to infinity, and from which ergodicity properties can be derived -- as well as solutions of 
counting problems (see e.g. \cite{MR2057305}). In a previous article (\cite[Lemma~3.5]{MR2290453}) the first author proved  this lemma in the rank one setting. 
For our purposes here we will need an analogous result  for the smaller and larger sets ${\mathcal O}^-_{r,c}(\xo,\gamma\xo)$ and ${\mathcal O}^+_{r,c}(\xo,\gamma\xo)$. 
One central point is the following elementary lemma originally due to  G.~Knieper (\cite[Lemma~4.1]{MR1465601}, see also \cite[Lemma~3.2]{MR2290453}):

\begin{lemma}\label{openposmass}\
Let $\delta> 0$,  $\mu$  a $\delta$-dimensional  $\Gamma$-invariant conformal density, $x\in\XX$ and $A\subseteq\Lim$ a $\Gamma$-invariant Borel set. Then $\mu_x(A)>0$ implies 
$\mu_x(O\cap A)>0$ for any open set $O\subset\rand$ with $O\cap \Lim\ne\emptyset$.
\end{lemma}
 \begin{proof} This follows again from the minimality of $\Lim$:\  Let $O\subset\rand$ be an open set with $O\cap \Lim\ne\emptyset$ and  
$\Lambda\subset\Gamma$ a finite set \st (\ref{coverLimset}) holds. Then 
$$A =A\cap \Lim \subseteq A\cap \bigcup_{\beta\in\Lambda} \beta O  =\bigcup_{\beta\in\Lambda}\beta (O\cap A)$$ 
by $\Gamma$-invariance of $A$, so  $\mu_x(O\cap A)=0$ would imply 
$$ \mu_x(A)
\le \sum_{\beta\in\Lambda} \mu_x\big(\beta (O\cap A)\big)=\sum_{\beta\in \Lambda} \mu_{\beta^{-1}x}(O\cap A) =0,$$
because $\mu_{\beta^{-1}x}$ is absolutely continuous with respect to $\mu_x$ for all $\beta\in\Lambda$.
\end{proof}

This lemma  allows us to prove the desired shadow lemma; the only difficulty is to bound from below the $\mu_{\xo}$-measure  of  the shadows ${\mathcal O}_{r,c}^-(\xo,\gamma\xo)$. At this point,  Proposition~\ref{smallshadowsarelarge} is crucial in our more general setting, the calculus being identical to the original proof (see Lemma~1.3 in~\cite{MR2057305}).

\begin{proposition}\label{shadowlemma}\
Let  $\delta>0 $  and $\mu$ a  $\delta$-dimensional $\Gamma$-invariant 
conformal density. Then for any $r>0$ there exists a
constant $c_0\ge r$ with the following property: If $c\geq c_0$ there
exists a constant $D=D(c)>1$ \st for all $\gamma\in\Gamma$ with
$d(\xo,\gamo)>2c$ we have
$$ \frac1{D}\; \e^{-\delta d(\xo,\gamma \xo)}\le \mu_\xo\big({\mathcal O}_{r,c}^-(\xo,\gamma\xo)\big)\le \mu_\xo\big(\pr_\xo\big(B_{\gamo}(c)\big)\le \mu_\xo\big({\mathcal O}_{c,c}^+(\xo,\gamo)\big)\le D \e^{-\delta d(\xo,\gamma\xo)}.$$ \end{proposition}

\begin{proof} The two inequalities in the middle are obvious from~(\ref{shadrelation}).  We keep the notation from the proofs of Proposition~\ref{lrcinproduct} and \ref{smallshadowsarelarge}. Fix  $r>0$ and choose an open neighborhood  $O\subset \rand$  of $h^+$ and $c_0\ge r$,  $c_0>\width(h)$ according  to Proposition~\ref{smallshadowsarelarge}. Hence for all $c\ge c_0$ and 
any $\gamma\in\Gamma$ with $d(\gamma\xo,\xo)>r+c$ there 
exists $\beta=\beta(\gamma)\in \{h^{M}, g^M h^{M}\}$ such that  
\[ \beta O\subseteq {\mathcal O}_{r,c_0}^-(\gamma\xo,\xo)\subseteq {\mathcal O}_{r,c}^-(\gamma\xo,\xo).\]  
From  Lemma~\ref{openposmass} we know that $q:=\min\{\mu_\xo(h^M O), \mu_\xo (g^Mh^M O)\}$ is positive and depends only on $O$, $M$ and the isometries $h$ and $g$ even if $\beta$ depends on $\gamma$. Hence for all $\gamma\in\Gamma$ with $d(\xo,\gamma\xo)>2c$  
$$  \mu_{\xo}\big({\mathcal O}_{r,c}^-(\gamma\xo,\xo)\big)\geq \mu_{\xo}(\beta O )\geq q>0$$
and
\be
\mu_\xo\big( {\mathcal O}^-_{r,c}(\xo,\gamma\xo)\big)&=&\mu_{\xo}\big(\gamma  {\mathcal O}^-_{r,c}(\gamma^{-1}\xo,\xo)\big)=\mu_{\gamma^{-1}\xo}\big( {\mathcal O}^-_{r,c}(\gamma^{-1}\xo,\xo)\big)\\
&\ge & \essinf_{\rand}
\Big(\frac{\d \mu_{\gamma^{-1} \xo}}{\d \mu_\xo}\Big)\; \mu_{\xo}\big({\mathcal O}_{r,c}^-(\gamma^{-1}\xo,\xo)\big)\geq \e^{-\delta d(\xo,\gamma\xo)}\, q.
\ee

The last inequality is straightforward: If $x,y\in \XX$ and  $\eta\in {\mathcal O}^+_{c,c}(x,y)$, considering a ray $\sigma_{z,\eta}$ from a point $z\in B_{x}(c)$ whose intersection with $B_y(c)$ is non-trivial 
we get
$$\bs_{\eta}(x,y)\geq d(x,y)-4c,$$
hence
\begin{align*}
\mu_\xo\big( {\mathcal O}^+_{c,c}(\xo,\gamma\xo)\big)&=\mu_{\gamma^{-1}\xo}\big( {\mathcal O}^+_{c,c}(\gamma^{-1}\xo,\xo)\big)\\
& \leq  \esssup_{{\mathcal O}^+_{c,c}(\gamma^{-1}\xo,\xo)}\Big(\frac{\d \mu_{\gamma^{-1} \xo}}{\d \mu_\xo}\Big)
\mu_\xo\big( {\mathcal O}^+_{c,c}(\xo,\gamma\xo)\big)
\leq \e^{4c\delta}\e^{-\delta d(\xo,\gamma\xo)}\mu_{\xo}(\rand).\end{align*}
\end{proof}

\begin{remark}\label{dgammasmall} If $d(\xo,\gamma\xo)\le 2c$, then ${\mathcal O}^+_{c,c}(\xo,\gamma\xo)=\rand$ and we have
\[ \mu_\xo\big( {\mathcal O}^+_{c,c}(\xo,\gamma\xo)\big)=\mu_\xo(\rand)\le \e^{2c \delta}\e^{-\delta d(\xo,\gamma\xo)}\mu_\xo(\rand).\]
So in particular the upper bound in Proposition~\ref{shadowlemma} holds for all $\gamma\in\Gamma$.
\end{remark}

\section{Properties of the radial limit set}\label{propradlimset} 

Recall that $\Gamma\subseteq\is(\XX)$ is a discrete isometry group of a Hadamard manifold $\XX$ which contains two rank one isometries $g$ and  $h$ with disjoint fixed point sets, 
and $\xo\in \Ax(h)$ is a fixed base point. 
In this section we will study an important subset of the limit set which is defined as follows.
For $c>0$ and $R\gg 1$ we set
$$ \Lim(c,R):= \bigcup_{\begin{smallmatrix}{\scriptscriptstyle\gamma\in\Gamma}\\{\scriptscriptstyle d(\xo,\gamma\xo)>R}\end{smallmatrix}}\pr_\xo(B_{\gamma\xo}(c)).$$
Obviously, these sets are non-decreasing in $c$ and non-increasing in $R$. 
Next we consider the decreasing limit of these sets as $R$ tends to infinity, i.e.  
\begin{equation}\label{radlimr}
\Lim(c):=  \bigcap_{R>0} \Lim(c,R).
\end{equation}
The {\hl radial limit set} $\radlim$ can now be defined as the increasing limit of the sets $\Lim(c)$ as $c$ tends to infinity, namely
\begin{equation}\label{radlimset}
\radlim:=\bigcup_{c>0} \Lim(c).
\end{equation}
It is independent of the origin $\xo\in\XX$ and will play a central role in the sequel. 

From here on fix  a  $\delta$-dimensional $\Gamma$-invariant conformal density for some $\delta>0$. Notice that by definition of a conformal density we have 
$0<\mu_\xo(\rand)<\infty$, and we will assume 
that $\mu_\xo$ is normalized \st $\mu_\xo(\rand)=1$.  
Concerning the $\mu_\xo$-measure of the radial limit set, we have the following easy lemma which is  
straightforward once we have the shadow lemma for projections. Even though the proof is exactly the same as in the case of CAT$(-1)$-spaces 
(compare statement (a) in \cite[p. 19]{MR2057305}), we include it here for the convenience of the reader. The converse statement is more involved and is Proposition~\ref{divseries} below.

\begin{lemma}\label{convseries}\ 
If $\ \sum_{\gamma\in\Gamma} \e^{-\delta  d(\xo,\gamo)}$ converges, then  $\mu_\xo(\radlim)=0$.
\end{lemma}

\begin{proof}    Suppose $\ \sum_{\gamma\in\Gamma} \e^{-\delta d(\xo,\gamo)}<\infty$ and $\mu_\xo(\radlim)>0$. Then by~(\ref{radlimset}) there exists a constant $C>0$ \st for all $c>C$ we have 
$$\mu_\xo\big(\Lim(c)\big)>0.$$
Fix $c>\max\{c_0,C\}$, where $c_0>0$ is the constant provided by the shadow lemma  Proposition~\ref{shadowlemma} (for $r>0$ arbitrary). Then by~(\ref{radlimr}) we have for any $R>0$ 
$$ 0<\mu_\xo\big(\Lim(c,R)\big)
\le \sum_{\begin{smallmatrix}{\scriptscriptstyle\gamma\in\Gamma}\\{\scriptscriptstyle d(\xo,\gamma\xo)>R}\end{smallmatrix}}\mu_\xo\big(\pr_\xo(B_{\gamma\xo}(c))\big)
\le  D(c) \sum_{\begin{smallmatrix}{\scriptscriptstyle\gamma\in\Gamma}\\{\scriptscriptstyle d(\xo,\gamma\xo)>R}\end{smallmatrix}} \e^{-\delta d(\xo,\gamma\xo)},
$$
hence the tail of the Poincar\'e series does not tend to zero. We conclude that\\
$\sum_{\gamma\in\Gamma} \e^{-\delta d(\xo,\gamma\xo)}$ diverges, in contradiction to the assumption.
\end{proof}

The following proposition states that $\Gamma$ acts ergodically on the radial limit set with respect to the measure class defined by $\mu$. 
\begin{proposition}\label{ergodicity}\ 
If $A\subset \radlim$ is a $\Gamma$-invariant Borel subset of $\radlim$, then
$\mu_\xo(A)=0$ or $\mu_\xo(A)=\mu_\xo(\rand)=1$.
\end{proposition}

The key step in the proof of this proposition is the following Lebesgue Besicovich' type derivation lemma whose proof  is very similar to the classical one in $\RR^n$ 
(see e.g. Chapter 11 in  \cite{MR2245472}). Here, balls are replaced by projections (called as well shadows) as in \cite{MR556586}. It is not clear in our setting whether
 or not projections are balls for a metric on $\rand$, and this is the main reason why such a lemma is needed (compare \cite{MR1465601}). Actually, the proof is the same as in the CAT$(-1)$ setting, but depends at several points on the shadow lemma. We give an account of this proof, since the arguments will be useful later on.

We recall the definition of the sets $\Lim(c)$ from (\ref{radlimr}).

\begin{lemma}\label{densitypoints}\  Fix $r>0$, let $c_0>0$ be the constant provided by Proposition~\ref{shadowlemma}  and let $c\ge c_0$. 
Then $\mu_\xo$-a.e. $\xi\in \Lim(c)$ is a density point, i.e. for any bounded Borel function 
$ \phi: \partial \XX\to \RR$ and $\mu_\xo$-a.e. $\xi \in \Lim(c)$  we have
$$\lim_{\begin{smallmatrix}{\scriptscriptstyle d(o,\gamma\xo)\to +\infty}\\{\scriptscriptstyle \xi\in \prsm_\xo(B_{\gamma\xo}(c))}\end{smallmatrix}}
\frac{1}{\mu_\xo\big(\pr_\xo(B_{\gamma\xo}(c))\big)}\int_{\pr_\xo(B_{\gamma\xo}(c))}\phi\; \d \mu_o=\phi(\xi).$$
\end{lemma}
\begin{proof}  We define  for $c\ge c_0$ and  $\gamma\in\Gamma$ with $d(\xo,\gamma\xo)>c$ sufficiently large  the mean value $M_\gamma^c(\psi)$ of a bounded Borel function $\psi:\rand\to\RR$ by 
\[ M_\gamma^c(\psi):=\frac1{\mu_\xo\big(\pr_\xo(B_{\gamma\xo}(c))\big)}\int_{\pr_\xo(B_{\gamma\xo}(c))}\psi\; \d \mu_o;\]
for $\mu_\xo$-a.e. $\xi\in \Lim(c)$ we further define the maximal function associated with $\psi$ by
\[ \psi^*(\xi):=\limsup_{\begin{smallmatrix}{\scriptscriptstyle d(o,\gamma\xo)\to +\infty}\\{\scriptscriptstyle \xi\in \prsm_\xo(B_{\gamma\xo}(c))}\end{smallmatrix}} M_\gamma^c(\psi).\]

Now let $\phi:\rand\to\RR$ be a bounded Borel function and $(\phi_n)$ a sequence of continuous functions whose limit $\mu_o$-a.e. and in $\LL^1(\mu_o)$ is $\phi$. 
The triangle inequality implies that for all $n\in\NN$ and $\gamma\in\Gamma$, $\xi\in\rand$ \st $\xi\in\pr_\xo(B_{\gamma\xo}(c))$ we have
\[ | M_\gamma^c(\phi)-\phi(\xi)|\le |\phi -\phi_n|^*(\xi)+ | M_\gamma^c(\phi_n)-\phi_n(\xi)|+|\phi_n-\phi|(\xi).\]
Since the lemma is obviously true for continuous functions, it only remains to prove the following Markov-type inequality:
\begin{equation}\label{statisticalcontrol}
 \mu_o(\{ \phi^*>\ep\})\leq \frac{C}{\ep}\int_{\partial \XX}| \phi|\,\d \mu_\xo.
\end{equation}
Its application to the bounded Borel function $\psi:=|\phi-\phi_n|$ then gives the conclusion of the lemma for any $\phi$ as above.

In order to prove inequality (\ref{statisticalcontrol}) we note that the set
\[ S(\psi,\ep):=\{\xi\in\Lim(c)\colon \psi^*(\xi)>\ep\}\] 
is infinitely covered by
\[ \bigcup_{ \gamma\in\widetilde\Gamma} \pr_\xo(B_{\gamma\xo}(c+1)),\quad \text{where }\ \widetilde\Gamma:=\{\gamma\in\Gamma\colon   d(\xo,\gamma\xo)\ge 2c+1, \ M_\gamma^c(\psi)>\ep\}. \]
The idea is to construct by induction a bounded multiplicity subcovering of $S(\psi,\ep )$. For that purpose we first denote by 
 $G:=G(c)\subset \widetilde \Gamma$ the set of elements $\gamma\in\widetilde \Gamma$ \st 
 \begin{align*}
 &\bigl( \gamma, \gamma'\in G,\  \gamma\neq \gamma'\ \Longrightarrow\ d(\gamma\xo,\gamma'o)\geq 1\bigr) \quad \text{and} \quad 
 \{ \psi^*>\ep \}\subseteq \bigcup_{\gamma\in G} \pr_\xo(B_{\gamma\xo}(c+1));
 \end{align*}
we set $  G_1:=\{\gamma\in G\colon  2c+1\leq d(\xo,\gamma \xo)<2c+2\}$ and then define for $k\geq 2$\\  
  $G_{k-1}':=G_1\cup\cdots\cup G_{k-1}\ $ and
\be G_k=\{ \gamma\in G  \colon  && \hspace{-5mm}  2c+k\leq d(\xo,\gamma \xo)<2c+k+1 \quad  \mbox{and}\\
&& \pr_\xo(B_{\gamma\xo}(c+1))\cap 
\pr_\xo(B_{\gamma'\xo}(c+1))=\emptyset\quad\, \forall\,  \gamma'\in G_{k-1}'\}.\ee
We now consider the subset $\Gamma^*\subset\widetilde\Gamma$ defined by
$$\Gamma^*:=\bigcup_{k\geq 1}G_k;$$
it is then a straightforward consequence of the triangle inequality that
  \begin{equation}\label{etoile}
  \bigcup_{\gamma\in \widetilde\Gamma}\pr_\xo(B_{\gamma\xo}(c+1))\subseteq \bigcup_{\gamma\in \Gamma^*}\pr_\xo(B_{\gamma\xo}(3c+3)).
  \end{equation}
  By construction, if $\gamma,\gamma'\in\Gamma^*$ satisfy
  \begin{equation}\label{intersectwithprime} \pr_\xo(B_{\gamma\xo}(3c+3))\cap \pr_\xo(B_{\gamma'\xo}(3c+3))\ne\emptyset,\end{equation}
then $\gamma,\gamma'$ belong to $G_k$ for the same $k$ and 
$d(\gamma'\xo,\gamma\xo)\leq 12c+13$ by the triangle inequality. The additional condition $d(\gamma\xo,\gamma'\xo)\ge 1$ -- which is necessary if we think of parabolic or mixed isometries in the sense of \cite[Definition~6.1]{MR823981} -- implies that there is a finite number of $\gamma'$ (depending on $c$ but not on $\gamma$ and $k$)  \st (\ref{intersectwithprime}) holds for any fixed $\gamma$. Consequently,
\begin{align*} 
   \mu_0(\{ \phi^*>\ep\})
   &\stackrel{(\ref{etoile})}{\leq} \sum_{\gamma\in \Gamma^*}\mu_\xo\big(\pr_\xo(B_{\gamma\xo}(3c+3))\big)\le    D'D \sum_{\gamma\in \Gamma^*}\mu_\xo\big( \pr_\xo(B_{\gamma\xo}(c))\big)\\
   &\leq \frac{D'D}{\ep}
 \sum_{\gamma\in \Gamma^*} \int_{\pr_\xo(B_{\gamma\xo}(c))} | \psi | \d \mu_\xo\le  \frac{MD'D}{\ep}\int_{\partial \XX}| \psi | \d \mu_\xo,\end{align*}
   where we used the shadow lemma Proposition~\ref{shadowlemma}, the fact that $M_\gamma^c(\psi)>\ep$ for $\gamma\in\Gamma^*$ and the finiteness of the subcovering. This is the desired inequality.\end{proof}

\begin{proof}[Proof of Proposition~\ref{ergodicity}.]  Let $A$ be a  $\Gamma$-invariant Borel subset of $\radlim$ such that\break $\mu_0(A)>0$. 
For $c\ge c_0$ sufficiently large the set $A\cap \Lim(c)$ is of positive $\mu_\xo$-measure, hence the above lemma applied to the characteristic function $\chi_{A}$ of $A$  gives
   $$\lim_{\begin{smallmatrix}{\scriptscriptstyle d(\xo,\gamma\xo)\to +\infty}\\{\scriptscriptstyle \xi\in \prsm_\xo(B_{\gamma\xo}(c))}\end{smallmatrix}}
\frac{\mu_\xo(\pr_\xo(B_{\gamma\xo}(c))\cap A)}{\mu_o\big(\pr_\xo(B_{\gamma\xo}(c)))}=1$$
(or equivalently with $A^c:=\rand\setminus A$ 
  $$\lim_{\begin{smallmatrix}{\scriptscriptstyle d(\xo,\gamma\xo)\to +\infty}\\{\scriptscriptstyle \xi\in \prsm_\xo(B_{\gamma\xo}(c))}\end{smallmatrix}}
\frac{\mu_\xo(\pr_\xo(B_{\gamma\xo}(c))\cap A^c)}{\mu_o\big(\pr_\xo(B_{\gamma\xo}(c)))}=0)$$
for $\mu_\xo$-a.e. $\xi\in A\cap \Lim(c)$.
Conformality~(\ref{conformality}) and  $\Gamma$-equivariance of $\mu$ imply that for any $\Gamma$-invariant subset $B$ of $\partial \XX$  we have
$$\mu_\xo(\pr_\xo(B_{\gamma\xo}(c))\cap B) \asymp \e^{\delta d(\xo,\gamma\xo)}\, \mu_\xo(\pr_{\gamma^{-1}\xo}(B_\xo(c))\cap B),$$
 where the symbol $\asymp$ means that one quantity is bounded from above and below by the second one up to  universal multiplicative constants. 
Applying this estimate to $B=A^c$ for the numerator and to $B=\partial \XX$ for the denominator gives 
  $$
\frac{\mu_\xo(\pr_{\gamma^{-1}\xo}(B_{\xo}(c))\cap A^c)}{\mu_o\big(\pr_{\gamma^{-1}\xo}(B_{\xo}(c)))}\asymp \frac{\e^{\delta d(\xo,\gamma\xo)}\mu_\xo(\pr_\xo(B_{\gamma\xo}(c))\cap A^c)}{\e^{\delta d(\xo,\gamma\xo)}\mu_o\big(\pr_\xo(B_{\gamma\xo}(c)))}=\frac{\mu_\xo(\pr_\xo(B_{\gamma\xo}(c))\cap A^c)}{\mu_o\big(\pr_\xo(B_{\gamma\xo}(c)))},$$
so that 
 $$\lim_{\begin{smallmatrix}{\scriptscriptstyle d(o,\gamma\xo)\to +\infty}\\{\scriptscriptstyle \xi\in \prsm_\xo(B_{\gamma\xo}(c))}\end{smallmatrix}}
\frac{\mu_\xo(\pr_{\gamma^{-1}\xo}(B_\xo(c))\cap A^c)}{\mu_\xo\big(\pr_{\gamma^{-1}\xo}(B_\xo(c))\big)}=0$$ 
for $\mu_o$-a.e. $\xi\in A\cap \Lim(c)$. If $c>c_0$ is big enough, then by Proposition~\ref{smallshadowsarelarge} there exists an open set $O\subseteq \rand$ whose intersection 
with $\Lim$ is non-trivial and which is contained in $\pr_{\gamma_j^{-1}\xo}(B_\xo(c))$ for a sequence  $(\gamma_j)\subset\Gamma$ such that $d(\xo,\gamma_j\xo)\to \infty$. 
Since  $\mu_\xo(\rand)=1$ and $O\cap A^c\subseteq \pr_{\gamma_j^{-1}\xo}(B_\xo(c))\cap A^c$ for all $j\in\NN$ we obtain
$$\mu_o(O\cap A^c)\le \limsup_{j\to\infty} \frac{\mu_\xo(\pr_{\gamma_j^{-1}\xo}(B_\xo(c))\cap A^c)}{\mu_\xo\big(\pr_{\gamma_j^{-1}\xo}(B_\xo(c))\big)}= 0.$$
But from the $\Gamma$-invariance of $A^c$ and Lemma~\ref{openposmass} we know that $\mu_\xo(A^c)>0$ would imply $\mu_\xo(O\cap A^c)>0$, so $\mu_\xo(A^c)=0$.
\end{proof}

Next we state a result concerning the atomic part of a $\delta$-dimensional  $\Gamma$-invariant conformal density with $\delta>0$ which will be needed later.
The proof is due to P.~J.~Nicholls (\cite[Theorem 3.5.3]{MR1041575}) in the case where $\XX$ is hyperbolic $n$-space. 
\begin{proposition}\label{atomicpart}\ 
A radial limit point cannot be a point mass for $\mu$.
 \end{proposition}
Before we give the proof, we will show the following two easy results concerning arbitrary point masses of a $\delta$-dimensional  $\Gamma$-invariant  conformal density with $\delta>0$. 
For $\eta\in\rand$ we denote $\Gamma_\eta:=\{\gamma\in\Gamma \colon \gamma\eta=\eta\}$ its stabilizer in $\Gamma$.
\begin{lemma}\label{stabi}
If $\eta$ is a point mass for  $\mu$, then for all $x\in\XX$ and all $\gamma\in\Gamma_\eta$ we have
$$\bs_\eta(x,\gamma x)=0.$$
\end{lemma}
\begin{proof} By $\Gamma$-equivariance and conformality~(\ref{conformality}) of $\mu$ we have
for all $x\in\XX$ and $\gamma\in\Gamma_\eta$
$$ 1=\frac{\mu_{x}(\gamma^{-1}\eta)}{\mu_x(\eta)}=\frac{\mu_{\gamma x}(\eta)}{\mu_x(\eta)}=\e^{\delta \bs_\eta (x,\gamma x)},$$
hence $\delta>0$ implies $\bs_\eta(x,\gamma x)=0$.\end{proof} 

\begin{lemma}\label{convatom}
Let $\eta$ be a point mass for  $\mu$, and $G\subset\Gamma$ a system of coset representatives 
for $\Gamma/\Gamma_\eta$.
Then the sum
$$ \sum_{g\in G} \e^{\delta \bs_\eta(\xo,g^{-1}\xo)}$$ 
converges. 
\end{lemma}
\begin{proof} If $g$ and $g'$ are distinct elements in $G$, then $g\eta\neq g'\eta$ by definition of $G$ and therefore
$$\sum_{g\in G} \mu_\xo(\{ g\eta\})\le \mu_\xo(\rand)<\infty. $$ 
We conclude by conformality~(\ref{conformality}) that\\[3mm]
$\hspace*{2.0cm}\displaystyle \sum_{g\in G} \e^{\delta \bs_\eta(\xo,g^{-1}\xo)}=\sum_{g\in G} \frac{\mu_{g^{-1}\xo}(\eta)}{\mu_\xo(\eta)}=\frac1{\mu_\xo(\eta)}\sum_{g\in G} \mu_{\xo}(\{ g\eta\})<\infty.$
\end{proof}

\begin{proof}[Proof of Proposition~\ref{atomicpart}.]  Let $\eta\in\radlim$ be a point mass for $\mu$. We are going to construct an infinite set $G\subset \Gamma$ of coset 
representatives for $\Gamma/\Gamma_\eta$ such that
$$\sum_{g\in G} \e^{\delta \bs_\eta(\xo,g^{-1}\xo)}$$ diverges. This gives a contradiction to Lemma~\ref{convatom} and thus proves the claim.

Since $\eta\in\radlim$ there exist $c>0$ and $(\gamma_j)\subset\Gamma$ \st  $\eta\in\pr_\xo(B_{\gamma_j\xo}(c))$ for all $j\in\NN$. Considering a geodesic ray emanating from $\xo$
which intersects $B_{\gamma_j\xo}(c)$ non-trivially we get  
$$ \bs_{\eta}(\xo,\gamma_j\xo)\ge d(\xo,\gamma_j\xo)-2c.$$
Hence passing to a subsequence if necessary we may assume that $\bs_{\eta}(\xo,\gamma_j\xo)$ is strictly increasing as $j$ tends to infinity. We claim that 
$G:=\{\gamma_j^{-1} \colon j\in\NN\}\subset\Gamma$ is the desired set.
Indeed, assume that there exist two distinct elements $\gamma_j^{-1},\gamma_l^{-1}\in G$  which represent the same coset in $\Gamma/\Gamma_\eta$. Then $\gamma_l\gamma_j^{-1}\in\Gamma_\eta$, and by 
Lemma~\ref{stabi} (with $x=\gamma_j\xo$)
$$0=\bs_{\eta}(\gamma_j\xo, (\gamma_l\gamma_j^{-1})\gamma_j\xo)=\bs_{\eta}(\gamma_j\xo, \gamma_l\xo).$$ 
From the cocycle identity of the Busemann function we conclude
$$\bs_{\eta}(\xo,\gamma_j\xo)=\bs_\eta(\xo, \gamma_j\xo)+\bs_\eta(\gamma_j\xo,\gamma_l\xo) \stackrel{(\ref{cocycle})}{=}\bs_\eta(\xo,\gamma_l\xo),$$
which contradicts the choice of our sequence $(\gamma_j)\subset\Gamma$. \end{proof} 
 
\section{Dynamical properties of the geodesic flow}\label{BMmeasure}

In this section we keep the previous notation. In particular, for $(\xi,\eta)\in\joinrand$ we recall that 
$\Pi^{-1}(\xi,\eta)$ is the set of vectors $v\in S\XX $ with negative end point $v^-=\xi$ and positive end point $v^+=\eta$. The set of foot points of such vectors 
$p(\Pi^{-1}(\xi,\eta))=(\xi\eta)\simeq C_{(\xi\eta)}\times\RR$  is a totally geodesic submanifold of $\XX$ and hence inherits the volume element from $\XX$.  We will denote this volume element $\vol_{(\xi\eta)}$ or sometimes simply $\vol$ 
 without making explicit reference to the submanifold $(\xi\eta)$. Similarly, the closed convex set $C_{(\xi\eta)}\subset\XX$ can be endowed with the volume element  induced from $\XX$ which we will denote by $\vol_{(\xi\eta)}^\perp$ or $\vol^\perp$. With Lebesgue measure  $\lambda$ on $\RR$ we clearly have
 \[ \vol_{(\xi\eta)}= \vol_{(\xi\eta)}^\perp\otimes \lambda.\]

Let $\mu$ be a  $\delta(\Gamma)$-dimensional  $\Gamma$-invariant conformal density   on $\rand$ normalized such that  $\mu_\xo$ is a probability measure. As in  Knieper's paper \cite{MR1652924}  we define a measure on the unit tangent bundle $S\XX$ of $\XX$ in the following way: For a Borel subset $E\subseteq S\XX$ we set
\begin{equation}\label{measuredef}
 m(E):=\int_{\joinrand} \d \mu_\xo(\xi)
\d \mu_\xo(\eta) \e^{2 \delta(\Gamma) \Gr_\xo(\xi,\eta)}\, \vol_{(\xi\eta)}\big(p(\Pi^{-1}(\xi,\eta)\cap E)\big),
\end{equation}
where 
$\Gr_\xo$ is the Gromov product~(\ref{GromovProd}) 
with respect to $\xo$ defined in Section~\ref{prelim}. This measure $m$ is obviously invariant by the geodesic flow; moreover, the conformality (\ref{conformality}) and the $\Gamma$-equivariance 
of $\mu$ imply that the measure $m$ does not depend on $\xo$ and is $\Gamma$-invariant. Hence it descends to a geodesic flow invariant measure $m_\Gamma$ on the 
unit tangent bundle $SM$ of the quotient $M=\XX/\Gamma$. In the following we denote $(g_\Gamma^t)_{t\in\RR}$  the geodesic flow of  $S\XX/\Gamma=SM$.

 In the sequel, $(\Omega,G,\nu)$ will be one of the following four dynamical systems (with the measure class of $\nu$
 invariant by  $G$ and with a Haar measure $\d g$ on $G$): 
 \[(\rand,\Gamma,\mu_\xo),\ (\joinrand, \Gamma, (\mu_\xo\otimes \mu_\xo)|_{\joinrand}),\ (\rand\times\rand, \Gamma, \mu_\xo\otimes\mu_\xo)\ \ \text{or }\; 
  (SM, (g_\Gamma^s)_{s\in\RR}, m_\Gamma).
 \]
A Borel set $E\subseteq \Omega$ is called {\hl wandering} if $\int_G \chi_{E} (g\omega)dg<\infty$  for $\nu$-a.e. $\omega\in E$ and {\hl  recurrent} otherwise. E.~Hopf's decomposition 
theorem (see for instance  \cite[p.~17]{MR797411}) asserts that the phase space $\Omega$  decomposes (uniquely up to sets of measure zero) into a disjoint union of $G$-invariant Borel sets $\Omega_D:=\cup_{k\in \ZZ}\Omega^k_D$   and $\Omega_C$ with $\Omega^k_D$  maximal wandering. $\Omega_D$ is called the {\hl dissipative part} and     $\Omega_C$ the  {\hl conservative part} of $\Omega$. The dynamical system is called {\hl (completely) conservative} if the dissipative part $\Omega_D$ has measure zero, and {\hl (completely) dissipative} if the conservative part $\Omega_C$ has measure zero.  
We prove that in our geometric context {\it either} complete conservativity {\it or} complete dissipativity occurs for the  Patterson-Sullivan measure $m_\Gamma$ on $SM=S\XX/\Gamma$.

It is easy to see how the 
following statement  can be deduced 
from the definition of the radial limit set~(\ref{radlimset}): 
\begin{lemma}\label{boundtang}\ 
For $ u\in S\XX$ we have
$u^+\in \radlim$ if and only if there exists a compact set $K\subset S\XX$ \st $\int_0^\infty \chi_{\Gamma K}( g^t u) \, \d t=\infty$. Hence $(SM, (g^s_\Gamma)_{s\in\RR}, m_\Gamma)$ 
is completely conservative if and only if $\mu_\xo(\radlim)=1$, and completely dissipative if and only if $\mu_\xo(\radlim)=0$.
\end{lemma}

It is one part of HTS dichotomy that $\Gamma$ is convergent if and only if $\mu_{\xo}(\radlim)=0$. The ``only if" part is Lemma~\ref{convseries};
the ``if'' part is more intricate for it needs a control of the multiplicity of the covering of the radial limit set by the non-increasing family $\Lim(c)$ defined in (\ref{radlimr}). The control is given  
 by  inequalities~(\ref{upbound}) and~(\ref{lowbound}) below. Unfortunately, due to the fact that the dimension of the submanifolds $(\xi\eta)$ can vary even if
$\xi$ and $\eta$ are in a small neighborhood of the repulsive respectively attractive fixed point of a weak rank one isometry, we need to impose the existence of a {\hl strong} rank one isometry in $\Gamma$ in order to get the lower bound~(\ref{lowbound}). 
For the proof of  Proposition~\ref{divseries} below -- which by Lemma~\ref{boundtang} implies the ``'if'' part -- we follow Roblin's exposition and therefore need to construct a 
compact Borel set $K_\Gamma$  satisfying the 
inequalities~(\ref{upbound}) and~(\ref{lowbound}) below.  This set will then  be recurrent (and of positive measure) if we assume $\Gamma$ to be divergent.

In order to construct such  a set $K_\Gamma$, we first define for $x\in \XX$ and $c>0$
\begin{equation}\label{K0def} K^0(x,c):=\{ v\in S\XX\colon  d(x,pv)<c\quad\text{and }\ pv\in H^0(v(x;v^-,v^+))\},\end{equation}
with $H^0$ as defined in~(\ref{Cxieta}), and $v(x;\xi,\eta)$ as in (\ref{projxxieta})  is the unique element $v\in S\XX$ whose foot point $pv$ is the orthogonal projection of $\xo$ to $(\xi\eta)$ and such that  $\Pi(v)=(\xi,\eta)$.
So $K^0(x,c)$ consists of all vectors $v\in SB_x(c)$  with foot point in the isometric image of  $C_{(v^-v^+)}$ passing through the orthogonal projection of $x$ to $(v^-v^+)$. In order to establish~(\ref{lowbound}), we need to ``thicken" these sets by the geodesic flow and consider
\begin{equation}\label{Kdef} K(x,c):=\{ g^s v\colon  v\in K^0(x,c),\ s\in (-c,c)\}\subseteq S B_x(2c) .\end{equation}
Such a set
 has the important property that the orbit of an arbitrary vector $v\in S\XX$ under the geodesic flow either does not intersect 
 it at all or spends precisely time $2c$ inside. 
In order to make the exposition of the proof of Proposition~\ref{divseries} below more transparent, we first state a few necessary and easy geometric estimates concerning the sets $K(x,c)$. 

For $c>0$ fixed we abbreviate $K:=\overline{ K(\xo,c)}$ the closure of $K(\xo, c)$ in $S\XX$ and 
 consider intersections of the form 
 \[K\cap g^{-t}\gamma K\qquad\text{and}\quad K\cap g^{-t}\gamma K\cap g^{-s-t}\varphi K\] 
in $ S\XX$ with $t,s>0$ and $\gamma,\varphi\in \Gamma$. As a  
direct consequence of the triangle inequality we obtain from    
\[K\cap g^{-t}\gamma K\ne \emptyset\]
the estimate  
\begin{eqnarray}\label{IT1} |d(\xo,\gamma\xo)-t|\leq 4c. \end{eqnarray}
 In the same manner, if 
\[K\cap g^{-t}\gamma K\cap g^{-s-t}\varphi K\ne \emptyset,\] then from (\ref{IT1}) we deduce that
\begin{eqnarray}\label{IT2} 0\leq d(\xo,\gamma\xo)+d(\gamma\xo,\varphi\xo)-d(\xo,\varphi\xo)\leq 12c.\end{eqnarray}

Moreover, we have the following relation between the open sets~(\ref{Kdef}) and the sets ${\mathcal L}_{c,c}(\xo,\gamma\xo)$ introduced in~(\ref{Lrc}).
\begin{lemma}\label{K0calL} 
\[\Pi\big(\{ K(\xo,c)\cap g^{-t}\gamma K(\xo,c)\colon t>0\}\big)={\mathcal L}_{c,c}(\xo,\gamma\xo)\]
\end{lemma}
\begin{proof} The inclusion ``$\subseteq$" follows from the first condition in the  definition~(\ref{K0def}) of the sets $K^0(x,c)$: If 
 $w\in  K(\xo,c)\cap g^{-t}\gamma K(\xo,c)$, then the geodesic 
 $\sigma_w$ satisfies
$d(\xo,\sigma_w(s))<c$  for some $s\in (-c,c)$, and $d(\gamma\xo, \sigma_w(t))<c$ for some $t>0$. Hence
\begin{eqnarray}\label{INC1} (w^-,w^+)\in{\mathcal L}_{c,c}(\xo,\gamma\xo).\end{eqnarray}

Conversely, if $(\xi,\eta)\in {\mathcal L}_{c,c}(\xo,\gamma\xo)$, there exists a geodesic $\sigma$ with $\sigma(-\infty)=\xi$, $\sigma(\infty)=\eta$ which first intersects $B_\xo(c)$ and then $B_{\gamma\xo}(c)$. Denote $v\in S\XX$ the unique vector such that $pv$ is the orthogonal projection of $\xo$ to $\sigma$, and $\sigma_v=\sigma$. Notice that if $\sigma$ is not a strong rank one geodesic, then $v$ is in general  different
from the orthogonal projection $v(\xo;\xi,\eta)$ to the whole set $(\xi\eta)$. Next let $\tau>0$ \st $\sigma_v(\tau)=p(g^\tau v)$ is the orthogonal projection of $\gamma\xo$ to the geodesic $\sigma$. In particular, we have
\[ d(\xo, pv)<c\quad\text{and}\quad d(\gamma\xo, p(g^\tau v))<c.\]
Moreover, since the geodesics determined by $v$ and $v(\xo;\xi,\eta)$ span a flat strip, and $pv$, $v(\xo;\xi,\eta)$ are orthogonal projections of the same point $\xo$, we have
\[  \bs_\xi(p v(\xo;\xi,\eta),pv)=\bs_\eta(pv(\xo;\xi,\eta),pv)=0.\]
This implies  $pv\in H^0(v(\xo;\xi,\eta))$ and $p(g^\tau v)\in H^\tau(v(\xo;\xi,\eta))= H^0(v(\gamma\xo;\xi,\eta))$, hence
\[ v\in K^0(\xo,c)\cap g^{-\tau}\gamma K^0(\xo,c)\subseteq K(\xo,c)\cap g^{-\tau}\gamma K(\xo,c)\]
for the particular $\tau>0$ from above.\end{proof}

As a direct consequence we obtain that for all $t,s>0$ and all $\gamma,\varphi\in \Gamma$ 
\begin{eqnarray}\label{INC2} \Pi\big(K\cap g^{-t}\gamma K\cap g^{-t-s}\varphi K\big)\subseteq
{\mathcal L}_{2c,2c}(\xo,\varphi\xo).\end{eqnarray}

Finally we remark that  if  $(\xi,\eta)\in{\mathcal L}_{2c,2c}(\xo,\varphi\xo)$, 
then there exists $z\in (\xi\eta)\cap B_{\xo}(2c)$ \st
\[ \Gr_{\xo}(\xi,\eta)=\frac12\big(\bs_\xi(\xo,z)+\bs_\eta(\xo,z)\big),\]
which immediately gives the estimate
\begin{equation}\label{GROMOV}  \Gr_{\xo}(\xi,\eta) \le 2c.\end{equation}

It remains to specify the constant $c>0$ for which the set $K:=\overline{K(o,c)}$ will satisfy the  
inequalities~(\ref{upbound}) and (\ref{lowbound}). Recall that $h\in\Gamma$ is a rank one element of $\Gamma$ and that the base point  $\xo\in\Ax(h)$ was chosen on the axis of $h$. We first fix $r=r_0>\width(h)$  and neighborhoods $U,V\subset\ganz$  of $h^-,h^+$ provided by Lemma~\ref{joinrankone} for $r_0$. Let $\Lambda\subset\Gamma$ be the finite symmetric set provided by Proposition~\ref{lrcinproduct}, set
\[ \rho:=\max\{d(\xo,\beta\xo)\colon \beta\in\Lambda\}\]
and -- with the constant $c_0>r$ from the shadow lemma Proposition~\ref{shadowlemma} -- fix
\[ c>c_0+\rho.\]
From here on 
\begin{equation}\label{Kdef}
K:=\overline{K(\xo,c)} 
\end{equation} will be defined with this constant $c$.

\begin{proposition}\label{divseries}\ 
If  $\,\Gamma$ contains a strong rank one isometry, then 
$\mu_\xo(\radlim)=0$ 
implies that 
$\ \sum_{\gamma\in\Gamma} \e^{-\delta(\Gamma) d(\xo,\gamo)}\,$ converges. 
\end{proposition}
\begin{proof}  We argue by contradiction, assuming that $\,\mu_\xo(\radlim)=0\ $ and that the series $\ \sum_{\gamma\in\Gamma} \e^{-\delta(\Gamma) d(\xo,\gamo)}$ diverges. We show that for the  compact set  $K\subset S\XX$  defined by (\ref{Kdef})
 the following inequalities hold for $T$ sufficiently  large with universal constants $C,C'>0$:

\begin{equation}\label{upbound}
\int_{0}^T \d t \int_0^T \d s \sum_{\gamma,\varphi\in\Gamma} m(K\cap g^{-t}\gamma K\cap g^{-t-s}\varphi K) \le  C
\biggl(\sum_{\begin{smallmatrix}{\scriptscriptstyle \gamma\in\Gamma}\\
        {\scriptscriptstyle d(o,\gamo)\le T}\end{smallmatrix}} \e^{-\delta(\Gamma)
    d(\xo,\gamma\xo)}\biggr)^2  
\end{equation}
\begin{equation}\label{lowbound}
\int_{0}^T \d t  \sum_{\gamma\in\Gamma} m(K\cap g^{-t}\gamma K) \ge  C' \sum_{\begin{smallmatrix}{\scriptscriptstyle\gamma\in\Gamma}\\
        {\scriptscriptstyle d(o,\gamo)\le T}\end{smallmatrix}} \e^{-\delta(\Gamma) d(\xo,\gamma\xo)}
\end{equation}

Once these inequalities are proved and under the assumption that the series $\sum_{\gamma\in\Gamma}\e^{-\delta(\Gamma)d(\xo,\gamma\xo)}$ diverges, one can apply the above mentioned 
generalization of the second Borel-Cantelli lemma and the conclusion follows as in  
\cite[p.~20]{MR2057305}: If $K_\Gamma\subset SM=S\XX/\Gamma$ denotes the projection of $K\subset S\XX$ to $S\XX/\Gamma$, then 
$$m_\Gamma\bigl(\{ v\in S M \colon \int_0^\infty \chi_{K_\Gamma\cap g^{-t}_\Gamma K_\Gamma}(v)=\infty\}\bigr)$$ 
is strictly positive, which means that the dynamical system 
$(S M, (g^s_\Gamma)_{s\in\RR}, m_\Gamma)$ is not completely dissipative. But by Lemma~\ref{boundtang} this is a contradiction to $\mu_\xo(\radlim)=0$.

We begin with the proof of (\ref{upbound}) which works even if $\Gamma$ contains only a weak rank one element; first we note   that 
$$F:=\max_{w\in  K} \bigl( \vol_{(w^-w^+)}^\perp (H^0(w) \cap p K)  \bigr)<\infty,$$ 
because $pK\subset \overline{B_\xo(2c)}$.  So if 
\[ (\xi,\eta)\in\Pi\big(K\cap g^{-t}\gamma K\cap g^{-t-s}\varphi K\big)\quad\text{for some }\   s,t>0,\]
then
\begin{align*}
&\int_0^T \d t \int_0^T \d s\; \vol_{(\xi\eta)} \big(p(\Pi^{-1}(\xi,\eta)\cap K\cap g^{-t}\gamma K\cap g^{-t-s}\varphi K)\big)\le (2c)^2\cdot F.
\end{align*}
By definition of the measure $m$, as a consequence of~(\ref{INC2}) and with the obvious relation
$${\mathcal  L}_{2c,2c}(\xo,\varphi\xo)\subseteq \rand\times {\mathcal O}^+_{2c,2c}(\xo,\varphi\xo)$$ 
we therefore get for  all 
$\gamma,\varphi\in \Gamma$ 
\begin{align*} \int_0^T \d t\int_0^T \d s  & \hspace{1mm} m(K\cap  g^{-t}\gamma K\cap g^{-t-s}\varphi K)\\
&  \le  4 c^2 F \int_{{\mathcal 
  L}_{2c,2c}(\xo,\varphi\xo)} \hspace{-1mm}\d \mu_\xo(\xi) \d \mu_\xo(\eta) \e^{2 \delta(\Gamma)
  \Gr_\xo(\xi,\eta)}  \\
& \stackrel{(\ref{GROMOV})}{ \le} 4c^2 F\cdot \e^{4c\delta(\Gamma)} \int_{{\mathcal
  L}_{2c,2c}(\xo,\varphi\xo)} \d \mu_\xo(\xi) \d \mu_\xo(\eta) \\
  &\le 4c^2 F\cdot \e^{4c\delta(\Gamma)}  \int_{\partial X}\d \mu_\xo(\xi)\int_{{\mathcal
  O}_{2c,2c}^+(\xo,\varphi\xo)}  \d \mu_\xo(\eta) \\
&= 4c^2 F\cdot \e^{4c\delta(\Gamma)}  \mu_\xo({\mathcal
  O}_{2c,2c}^+(\xo,\varphi\xo))\\
  &\le 4c^2 F\cdot \e^{4c\delta(\Gamma)}  D(2c) \cdot \e^{-\delta(\Gamma) d(\xo,\varphi\xo) },
  \end{align*}
  where we used the shadow lemma Proposition~\ref{shadowlemma} (together with Remark~\ref{dgammasmall}) in the last step.
We conclude that
\[ \int_0^T \d t\int_0^T \d s\; m(K\cap  g^{-t}\gamma K\cap g^{-t-s}\varphi K) \le C_1  \e^{-\delta(\Gamma) d(\xo,\varphi\xo) }\]
with the constant $C_1= 4c^2 F \cdot \e^{4c\delta(\Gamma)}D(2c)$ which only depends on $c$.

Making use repeatedly of (\ref{IT1}) and its consequence (\ref{IT2}), we further get
\begin{align*} \int_{0}^T \d t \int_0^T \d s &  \sum_{\gamma,\varphi\in\Gamma} m(K\cap
g^{-t}\gamma K\cap g^{-t-s}\varphi K)\\
&\quad \le   \sum_{\begin{smallmatrix}{\scriptscriptstyle\gamma,\varphi\in\Gamma}\\{\scriptscriptstyle d(o,\gamo)\le T+4c}\\{\scriptscriptstyle
          d(\gamo,\varphi\xo)\le T+4c}\end{smallmatrix}} C_1
    \e^{-\delta(\Gamma) (d(\xo,\gamo)+d(\gamo,\varphi\xo)-12c)}\\
&\quad  = C_1 \e^{12c\delta(\Gamma)} \sum_{\begin{smallmatrix}{\scriptscriptstyle\gamma,\alpha\in\Gamma}\\
        {\scriptscriptstyle d(o,\gamo)\le T+4c}\\{\scriptscriptstyle
          d(o,\alpha\xo)\le T+4c}\end{smallmatrix}}
    \e^{-\delta(\Gamma) (d(\xo,\gamo)+d(\xo,\alpha\xo))}\\
    &\quad \le  C_2 \biggl( \sum_{\begin{smallmatrix}{\scriptscriptstyle\gamma\in\Gamma}\\
        {\scriptscriptstyle d(o,\gamo)\le T+4c} \end{smallmatrix}}
    \e^{-\delta(\Gamma) d(\xo,\gamo)}\biggr)^2,\end{align*}
where  $C_2$ is again  a constant depending only on $c$.
    
Since \[\displaystyle \sum_{T<d(\xo,\gamo)\le T+4c} \e^{-\delta(\Gamma) d(\xo,\gamo)}\quad \]
is uniformly bounded in $T$ as a direct consequence of  Corollary~3.8 in 
\cite{MR2290453}, we have established (\ref{upbound}).\\

It remains to prove inequality~(\ref{lowbound}), where we will have to require that  $h$ is a strong rank one element in $\Gamma$. We recall that under this assumption (and with the notation introduced before (\ref{Kdef})) every pair of points $(\xi,\eta)\in U\times V$ can be joined by a strong rank one geodesic.

Using the definition of $m$, Lemma~\ref{K0calL}
and the  non-negativity of the Gromov product,  we first obtain for $\gamma\in\Gamma$ with $d(\xo,\gamma\xo)>4c$
\begin{align*} & \hspace{-0.3cm}\int_0^T \d t\,  m(K\cap g^{-t}\gamma K)\\ 
&=   \int_{0}^{T} \d t \int_{\joinrand} \d \mu_\xo(\xi)
\d \mu_\xo(\eta) \e^{2 \delta(\Gamma) \Gr_\xo(\xi,\eta)}\vol_{(\xi\eta)}\bigl(p(\Pi^{-1}(\xi,\eta)\cap K\cap g^{-t}\gamma K)\bigr)\\
&\ge \int_{{\mathcal 
    L}_{c,c}(\xo,\gamo)} \d \mu_\xo(\xi)\d \mu_\xo(\eta)\biggl(\int_0^T \d t\, \vol_{(\xi\eta)}\bigl(p(\Pi^{-1}(\xi,\eta)\cap K\cap g^{-t}\gamma K)\bigr)\biggr).
\end{align*}
Here the problem appears that we cannot in general uniformly  bound from below the integral \[\int_0^T \d t\, \vol_{(\xi\eta)}\big(p(\Pi^{-1}(\xi,\eta)\cap K\cap g^{-t}\gamma K)\big),\]
since the geodesics intersecting $B_\xo(c)$ and $B_{\gamma\xo}(c)$ may belong to totally geodesic submanifolds 
$ (\xi\eta)\simeq C_{(\xi\eta)}\times \RR$ with $\vol_{(\xi\eta)}^\perp  \bigl(C_{(\xi\eta)}\bigr)$ arbitrarily small (compare our Remark~\ref{zerowidthnotregular}).
However,  we know from Lemma~\ref{joinrankone} that when restricting to pairs of points $(\xi,\eta)\in U\times V$, then every 
\[ w\in \{\Pi^{-1}(\xi,\eta)\cap K\cap g^{-t}\gamma K \colon  t>0\}\]
determines  a {\hl strong} rank one geodesic  $\sigma_w$ which spends a time $2c$ in $pK$ and later the same time $2c$ in $p(\gamma K)$. Moreover, we have 
$(\xi\eta)=\sigma_w(\RR)\simeq \RR$ (hence in particular $\vol_{(\xi\eta)}$ is Lebesgue measure $\lambda$ on $\RR$), so  if
$(\xi,\eta)\in {\mathcal L}_{c,c}(\xo,\gamma\xo)\cap (U\times V)$ and  $T>d(\xo,\gamma\xo)+4c$, then 
\[ \int_0^T \d t\, \vol_{(\xi\eta)}\big(p(\Pi^{-1}(\xi,\eta)\cap K\cap g^{-t}\gamma K)\big)= 2c,\]
Since $\Gamma$   acts by isometries, the same is true for 
\[(\xi,\eta)\in  {\mathcal L}_{c,c}(\xo,\gamma\xo)\cap \beta (U\times V)\] with $\beta\in\Gamma$ arbitrary.  We conclude that for any $\beta\in \Lambda$
\begin{align*} &\int_0^T \d t\,  m(K\cap g^{-t}\gamma K)\\ 
&\ge \int_{{\mathcal 
    L}_{c,c}(\xo,\gamo)\cap\beta (U\times V)} \d \mu_\xo(\xi)\d \mu_\xo(\eta)\underbrace{\biggl(\int_0^T \d t\, \vol_{(\xi\eta)}\big(p(\Pi^{-1}(\xi,\eta)\cap K\cap g^{-t}\gamma K)\big)\biggr)}_{= 2c}\\
&= 2c\cdot \int_{\beta\bigl( \beta^{-1} {\mathcal 
    L}_{c,c}(\xo,\gamo)\cap (U\times V)\bigr)} \d \mu_\xo(\xi)\d \mu_\xo(\eta).
\end{align*}
Finally, by Proposition~\ref{lrcinproduct} we know that for all $\gamma\in\Gamma$ with  $d(\xo,\gamma\xo)>R\,$ (with $ R>4c$ sufficiently large)
there exists an element $\beta$ in the finite set 
$\Lambda\subset\Gamma$ with the property 
$$ {\mathcal L}_{r,c}(\xo,\beta^{-1}\gamma \xo)\cap \big(U\times V\big)\supseteq (U \cap \rand) \times {\mathcal O}_{r,c}^-(\xo,\beta^{-1}\gamma\xo);$$
using (\ref{lrcinclus}), $c_0\ge r$ and $c>c_0+\rho$ we also have the inclusion
\[ {\mathcal L}_{r,c}(\xo,\beta^{-1}\gamma \xo)=\beta^{-1}{\mathcal L}_{r,c}(\beta\xo,\gamma\xo)\subseteq
\beta^{-1} {\mathcal L}_{r+\rho,c}(\xo,\gamma\xo)\subseteq \beta^{-1} {\mathcal L}_{c,c}(\xo,\gamma\xo). \]
Summarizing we obtain for all $\gamma\in\Gamma$ with $d(\xo,\gamma\xo)\in (R,T-4c)$
\begin{align*} 
\int_0^T \d t\,  m(K\cap g^{-t}\gamma K)&\ge 2c\cdot  \int_{\beta \big((U\cap \rand)\times{\mathcal O}_{r,c}^-(\xo,\beta^{-1}\gamo)\big)} \d \mu_\xo(\xi)\d \mu_\xo(\eta)\\
&= 2c  \cdot \mu_{\xo}(\beta U) \mu_\xo(\beta {\mathcal O}_{r,c}^-(\xo,\beta^{-1}\gamma\xo) )\\ 
&\ge 2c  \cdot \mu_{\xo}(\beta U) \e^{-\delta(\Gamma)d(\xo,\beta^{-1}\xo)}\mu_\xo( {\mathcal O}_{r,c}^-(\xo,\beta^{-1}\gamma\xo) )\\ 
&\ge 2c  \cdot \mu_{\xo}(\beta U) \e^{-\delta(\Gamma)d(\xo,\beta^{-1}\xo)} \cdot D(c)\e^{-\delta(\Gamma)d(
\xo,\beta^{-1}\gamma\xo)}\\ 
& \ge 2c\cdot  \min_{\beta\in\Lambda} \mu_\xo(\beta U) \cdot \e^{-2 \delta(\Gamma)\rho} D(c)  \e^{-\delta(\Gamma) d(
\xo,\gamma\xo)}\\
&= C' \e^{-\delta(\Gamma) d(\xo,\gamma\xo)}
\end{align*}
with a constant $C'$ depending only on $c$ and the fixed finite set $\Lambda\subset\Gamma$; in the last three inequalities we used the $\Gamma$-equivariance and the conformality~(\ref{conformality}) of $\mu$, the shadow lemma Proposition~\ref{shadowlemma} and the triangle inequality for the exponent.

Finally, taking the sum over all elements  $\gamma\in\Gamma$ with $d(\xo,\gamma\xo)\in (R,T-4c)$
 we get 
 $$ \int_0^T  \sum_{\gamma\in\Gamma} m(K\cap g^{-t}\gamma K)\;\d t
\ge C'
\sum_{\begin{smallmatrix}{\scriptscriptstyle\gamma\in\Gamma}\\{\scriptscriptstyle
      R< d(o,\gamo)\le T-4c}\end{smallmatrix}}
\e^{-\delta(\Gamma) d(\xo,\gamo)},$$
and inequality~(\ref{lowbound}) follows with the same argument as above, namely that 
$$\sum_{\begin{smallmatrix}{\scriptscriptstyle\gamma\in\Gamma}\\{\scriptscriptstyle
       d(o,\gamo)\le R}\end{smallmatrix}}
\e^{-\delta(\Gamma) d(\xo,\gamo)}<\infty \qquad\text{and}\quad \sum_{\begin{smallmatrix}{\scriptscriptstyle\gamma\in\Gamma}\\{\scriptscriptstyle
      T-4c< d(o,\gamo)\le T}\end{smallmatrix}}
\e^{-\delta(\Gamma) d(\xo,\gamo)}$$ 
is uniformly bounded in $T$.\end{proof}

\section{Ergodicity for divergent groups}\label{HopfArgument}

In this  section we again assume that $\Gamma\subseteq\is(\XX)$  contains a {\hl strong} rank one isometry
 $h$ and that the base point  $\xo\in \Ax(h)$ belongs to the axis of $h$. As in \cite{MR1652924} we let $d_1$ be the metric on the unit tangent bundle $S\XX$ defined by
$$ d_1(u,v):=\max \{ d(p g^t u,p g^t w) \colon 0\le t\le 1\}\ \mbox{ for} \ u,v\in S\XX .$$ 
Moreover, a vector $v\in S\XX$ is called {\hl recurrent}, if there exist sequences $(\gamma_n)\subset\Gamma$ and $(t_n)\to\infty$ \st $\gamma_n g^{t_n}v$ converges to $v$ in $S\XX$.

In order to prove ergodicity of the geodesic flow on $SM$ with respect to the measure $m_\Gamma$ we will use the famous  Hopf argument (see \cite{Hopf}, \cite{MR0284564}) as in \cite{MR2057305} and \cite{MR1652924}.
In our setting we need Knieper's Proposition~4.1 which holds only  under the stronger assumption that $\Gamma$ contains a strong rank one element. We recall only its statement here:
\begin{proposition}\label{KniepersProp}(\cite{MR1652924},  Proposition~4.1) \ 
\  Let $v\in S\XX$ be a recurrent vector defining a strong rank one geodesic $\sigma_v$ in $\XX$. Then for all $u\in S\XX$ with $u^+=v^+$ and $\bs_{v^+}(pv,pu)=0$ we have
$$ \lim_{t\to\infty} d_1(g^t v, g^tu)=0.$$
\end{proposition}

\begin{proposition}\label{conservative}\ 
If $(SM, (g^t_\Gamma)_{t\in\RR}, m_\Gamma)$ is completely conservative, that is if\break $\mu_o(\radlim)=1$, 
then it is ergodic.
\end{proposition}
\begin{proof}  Applying Hopf's generalization of the Birkhoff ergodic theorem (\cite{Hopf}) we will show that for some strictly positive function $\rho\in \LL^1(m_\Gamma)$ and every function $f\in \LL^1(m_\Gamma)$ the limit 
\[ \lim_{|T|\to\infty} \frac{\int_0^T  f(g_\Gamma^{t}(u))\d t}{\int_0^T \rho(g_\Gamma^{ t}(u))\d t}\] 
exists and is constant $m_\Gamma$-almost everywhere:

As in \cite{MR556586} we first introduce a function $\tilde\rho: S\XX\to \RR\,$ by
$$ \tilde\rho(u)= \e^{-4\delta(\Gamma) d(pu,\Gamma\xo)}\quad\text{for }\  u\in S\XX,$$
which descends to a function $\rho$ on $S\XX/\Gamma$. Furthermore, since for any $(\xi,\eta)\in\joinrand$ and $R\ge 1$  
$$\vol_{(\xi\eta)} (p\big(\Pi^{-1}(\xi,\eta)\big)\cap B_\xo(R) )\le \const\cdot R^{k}$$
with $k\le \dim X -1$,   
and $p\big(\Pi^{-1}(\xi,\eta)\big)\cap B_\xo(R)\neq \emptyset$ implies $\Gr_\xo(\xi,\eta)< R$, we have
$$ m(SB_\xo(R))\le  \const\cdot R^{k} \e^{2\delta(\Gamma) R}.$$
If $D_\Gamma$ denotes the Dirichlet domain for $\Gamma$ with center $\xo$, we have $\tilde\rho(u)=\e^{-4\delta(\Gamma) d(pu,\xo)}$ for all $u\in S\overline{D_\Gamma}$. 
Setting $W(R):=(SB_\xo(R)\setminus SB_\xo(R-1))\cap S\overline{D_\Gamma}$ we therefore have
\begin{align*}
 \int_{W(R)} \tilde\rho(u)\d m(u) & \le  \e^{-4\delta(\Gamma) (R-1)}   \int_{SB_\xo(R)} \d m(u) \le \const\cdot R^{k} \e^{-2\delta(\Gamma) R}\\
 & \le \const\cdot   \e^{-\delta(\Gamma) R},\end{align*}
hence $\rho\in \LL^1(m_\Gamma)$. Moreover $\tilde\rho$  
is continuous since  for any $u,v\in S\XX$ with $d(pu,pv)\le 1$ we have
\begin{eqnarray*}
|\tilde\rho(u)-\tilde\rho(v)| &\le & \e^{-4\delta(\Gamma) d(pu,pv)} \max \{1- \e^{-4\delta(\Gamma) d(pu,pv)}, \e^{4\delta(\Gamma) d(pu,pv)}-1\}\\
&\le &\tilde\rho(u) \cdot 4\delta(\Gamma) d(pu,pv)\e^{4\delta(\Gamma)}.
\end{eqnarray*}

Next let $f\in\Cnt_c(SM)$ and $\tilde f\in \Cnt(S\XX)$ be a lift to $S\XX$. 
Since $(SM, (g^t_\Gamma)_{t\in\RR}, m_\Gamma)$ is completely conservative, Hopf's individual ergodic theorem (see \cite{Hopf}, p.~53) states that for $m$-almost every $u\in S\XX$ the limits
$$  \tilde f^\pm(u)=\lim_{T\to +\infty} \frac{\int_0^T \tilde f(g^{\pm t}(u))\d t}{\int_0^T \tilde \rho(g^{\pm t}(u))\d t}$$ 
exist and are equal.   

For the remainder of the proof we follow the argument of G.~Knieper in \cite{MR1652924}: The Hopf argument can be applied locally and then extended by a transitivity argument. 
Let $h^+$, $h^-$ denote the attractive and repulsive fixed point of the strong rank one isometry $h\in \Gamma$.   By Lemma~\ref{joinrankone}  there exist open neighborhoods $U$ of $h^-$ and $V$ of $h^+$  
in $\ganz$ \st each $\xi\in U$ and $\eta\in V$ can be joined by a unique strong rank one geodesic (up to parametrization). We define 
\begin{align*}
\Omega(U,V)&:=\{w\in S\XX \colon w^-\in U,\ w^+\in V\},\\
 \Omega^{\small{\text{rec}}}
 (U,V) &:=\{w\in \Omega(U,V) \colon \ w\ \text{is recurrent}\}.
 \end{align*}
By Poincar{\'e} recurrence (see e.g. \cite[Satz 13.2]{Hopf}) -- which holds because\break  $(SM, (g^t_\Gamma)_{t\in\RR}, m_\Gamma)$ is conservative and because the topological space $\Omega(U,V)/\Gamma$ has a countable base -- the set $\Omega^{\small{\text{rec}}}(U,V)$
has full measure in $\Omega(U,V)$ with respect to $m$. From the previous paragraph we further know that the set 

 $$ \Omega'(U,V):=\{ w\in \Omega^{\small{\text{rec}}}(U,V)\colon\ \tilde f^+(w)=\tilde f^-(w)\} $$
again has full measure in $\Omega(U,V)$. The local product structure of $m$  implies the existence of a vector $w\in\Omega'(U,V)$ with $w^-=\xi\,$ 
\st
$$ G_\xi:=\{\eta\in V \colon  \exists\, u\in  \Omega'(U,V) \quad\mbox{with }\ u^-=\xi,\ u^+=\eta\}$$ 
has full measure in $V$ \wrt $\mu_\xo$.

\begin{lemma}\ 
$\tilde f^+$ is constant $m$-a.e. on $\Omega'(U,V)$.
\end{lemma}
\begin{proof}
For $m$-a.e. $v\in \Omega'(U,V)$ we have $v^+\in G_\xi$. Let $u\in\Omega'(U,V)$ be  \st $u^-=\xi=w^-$ and $u^+=v^+$. Notice that by definition of $\Omega'(U,V)$ the vector $u$ is recurrent, defines a strong rank one geodesic and satisfies $\tilde f^+(u)=\tilde f^-(u)$. If $s_1$, $s_2\in\RR$ are chosen such 
that $\bs_{v^+}(pv, p g^{s_1}u)=\bs_{\xi}(pw,p g^{s_2} u)=0,$ then Proposition~\ref{KniepersProp} implies 
$$\lim_{t\to\infty} d_1(g^{t} g^{s_1}u, g^{t} v)= \lim_{t\to\infty}d_1( g^{-t}g^{s_2} u ,g^{-t}w)=0.$$
It now follows from the continuity of $\tilde f$, the $(g^t)$-invariance of $\tilde f^+$ and $g^{s_2}u, w \in \Omega'(U,V)$ that \\[-1mm] 

$\qquad \tilde f^+(v)=\tilde f^+(g^{s_1}u)= \tilde f^+(g^{s_2}u)=\tilde f^-(g^{s_2}u)=\tilde f^-(w)=\tilde f^+(w).$\end{proof}

The previous lemma together with Proposition~\ref{KniepersProp} implies that $\tilde f^+$ is constant  $m$-a.e. on the set of vectors $v\in S\XX$ with
$$v^+\in V':=\{ \eta\in V \colon   \exists\, u\in  \Omega'(U,V) \quad\mbox{with }\ u^+=\eta\}.$$
Next we consider the set\vspace{-2mm}
 $$Y:=\bigcup_{\gamma\in\Gamma} \gamma V'.$$
We have  $\mu_\xo(Y\cap V)= \mu_\xo(V)$, and therefore $\mu_x(Y\cap V)=\mu_x(V)$ for all $x\in\XX$. Furthermore, every $\zeta\in \rand$ possesses 
an open neighborhood $O(\zeta)\subseteq\rand$ that can be mapped into $V$ by an element of $\Gamma$. We conclude that $\mu_\xo(Y\cap O(\zeta))=\mu_\xo(O(\zeta))$ 
for each $\zeta\in\rand$, hence $\mu_\xo(Y)=\mu_\xo(\radlim)$. Consequently, the set 
\[ Z:=\{v\in S\XX \colon v^+\in Y\}\]
 has full measure \wrt $m$, 
hence $\tilde f^+$ is constant $m$-a.e. on $Z$ by the $\Gamma$-invariance of $\tilde f^+$. 

We conclude that for arbitrary $f\in \Cnt_c(SM)$  the ergodic means
$$ \frac{\int_0^T f\circ g_\Gamma^t\;\d t}{\int_0^T  \rho\circ g_\Gamma^t\;\d t}$$ converge $m_\Gamma$-almost everywhere to a constant function as $|T|\to\infty$. 
Since $\Cnt_c(SM)$ is dense in $\LL^1(m_\Gamma)$ this remains true for each function $f\in \LL^1(m_\Gamma)$, which proves that
 $(SM,(g^t_\Gamma)_{t\in\RR}, m_\Gamma)$ is ergodic.\end{proof}

\begin{corollary}\ 
If $\mu_\xo(\radlim)=1$, then the set of vectors in $SM$ defining strong rank one geodesics has full measure \wrt $m_\Gamma$.
\end{corollary}
We now sum up all the previous results in the following statement, which is the main theorem from the introduction:
\begin{theorem}\label{HTS}\ 
Suppose $\Gamma\subset\is(\XX)$ is a non-elementary discrete group which contains a strong rank one isometry. Let $\mu$ be a $\delta(\Gamma)$-dimensional conformal density 
normalized such that $\mu_\xo(\rand)=1$, 
and $m_\Gamma$ the associated Patterson-Sullivan measure on $SM=S\XX/\Gamma$. Then exactly one of the following two complementary cases holds, and the statements (i) to (iv) are equivalent in each 
case: \\[2mm]
1.~Case:
\begin{enumerate}
\item[(i)] $\sum_{\gamma\in\Gamma} \e^{-\delta(\Gamma)d(\xo,\gamma\xo)}$ converges.
\item[(ii)] $\mu_\xo(\radlim)=0$.
\item[(iii)] $(SM, (g_\Gamma^t)_{t\in\RR}, m_\Gamma)$ is completely dissipative.
\item[(iv)] $(SM, (g_\Gamma^t)_{t\in\RR}, m_\Gamma)$ is  non-ergodic.
\end{enumerate}
2.~Case:
\begin{enumerate}
\item[(i)] $\sum_{\gamma\in\Gamma} \e^{-\delta(\Gamma)d(\xo,\gamma\xo)}$ diverges.
\item[(ii)] $\mu_\xo(\radlim)=1$.
\item[(iii)] $(SM, (g_\Gamma^t)_{t\in\RR}, m_\Gamma)$ is completely conservative.
\item[(iv)] $(SM, (g_\Gamma^t)_{t\in\RR}, m_\Gamma)$ is  ergodic.
\end{enumerate}
\end{theorem}

\vspace{0.0cm}
\section{The space of geodesics of $\XX$}
\label{surfaces}

As we mentioned in the introduction, the action of the geodesic flow can be interpreted as an action of $\Gamma$ on the space $\partial^2\XX$ of end point pairs of geodesics. From the construction of the measure $m_{\Gamma}$ we immediately get the following 
\begin{lemma}\label{ergodicequiv}\ 
$(SM, (g^t_\Gamma)_{t\in\RR}, m_\Gamma)$ is ergodic if and only if $(\joinrand, \Gamma, (\mu_\xo\otimes\mu_\xo)\ein_{\joinrand})$ is ergodic.
\end{lemma}
So in particular, the dynamical system $(\joinrand, \Gamma, (\mu_\xo\otimes\mu_\xo)\ein_{\joinrand})$ is ergodic and  completely conservative when $\Gamma$ is divergent. 
Unfortunately, in the rank one setting it is not clear whether in the first case of Theorem~\ref{HTS} we have  
complete dissipatitivity of 
$(\joinrand, \Gamma, (\mu_\xo\otimes\mu_\xo)\ein_{\joinrand})$. However, if 
\[\partial^2_1 \XX :=\{(\sigma(-\infty),\sigma(+\infty))\in\joinrand\colon \sigma\ \text{is a (weak) rank one geodesic}\} \]
denotes the $\Gamma$-invariant subset of end point pairs of 
geodesics without 
flat half plane, 
we have the following
\begin{lemma}\label{muradlim0}\ 
If $\mu_\xo(\radlim)=0$ then $(\partial_1^2\XX, \Gamma, (\mu_\xo\otimes\mu_\xo)\ein_{\partial_1^2\XX})$ is completely dissipative.
\end{lemma}
\begin{proof} The idea of proof due to Sullivan is to decompose the phase space into a countable disjoint union of wandering sets indexed by finite subsets $\Lambda\xo$ of the orbit $\Gamma\xo$ (see also the proof of (b) in \cite[p.~19]{MR2057305}):

For each pair of points $(\xi,\eta)\in\partial^2_1\XX$ with $\xi\notin\radlim$ and $\eta\notin\radlim$ we define the set
$$\Gamma_{\xi,\eta}:=\{\gamma\in\Gamma\colon d(\gamma\xo,(\xi\eta))\ \mbox{is minimal}\},$$
which  is non-empty and finite by discreteness of $\Gamma$ and  the fact that the boundary of the closed convex set $(\xi\eta)$ equals the set $\{ \xi,\eta\}$ (and hence does not contain any radial limit points).  
For every finite set $\Lambda\subset\Gamma$ we denote $E_\Lambda\subset\partial^2_1\XX$ the Borel set
$$ E_\Lambda:=\{(\xi,\eta)\in\partial^2_1\XX\colon \Gamma_{\xi,\eta}=\Lambda\}.$$
If $\gamma E_\Lambda\cap E_\Lambda\neq \emptyset$, then $\gamma \Lambda=\Lambda$. Furthermore, since the stabilizer of $\Lambda$ in $\Gamma$ is finite, $E_\Lambda$ is wandering 
(by definition). 
Since the union of finite subsets of $\Gamma$ is countable, we obtain that the set $Y:=\{(\xi,\eta)\in\partial^2_1\XX\colon \xi,\eta\notin\radlim\}$ is contained in the dissipative 
part of $\partial^2_1\XX$. By definition of the product measure we further have 
\[ (\mu_\xo\otimes\mu_\xo)\ein_{\partial^2_1\XX}(\partial_1^2\XX\setminus Y)\le  (\mu_\xo\otimes\mu_\xo)\ein_{\joinrand}(\partial_1^2\XX\setminus Y)=0.\] 
Since the conservative part of the dynamical system $(\partial_1^2\XX, \Gamma, (\mu_\xo\otimes\mu_\xo)\ein_{\partial_1^2\XX})$ is included in $ \partial_1^2\XX\setminus Y$, complete 
dissipativity follows. 
\end{proof}

Notice that the proof above does not work in general for the dynamical system $(\partial^2\XX, \Gamma, (\mu_\xo\otimes\mu_\xo)\ein_{\partial^2\XX})$: If $(\xi,\eta)\in \joinrand$ and $\xi,\eta\notin\radlim$, then the set $(\xi\eta)$ may accumulate on $\partial X$ on other points than $\xi$ and $\eta$. 
For instance, a flat cylinder isometrically embedded in $M$ unfolds in the universal covering $X$ as a totally geodesic 
flat submanifold whose boundary is a circle containing two radial limit points. So if $\xi,\eta$ are two antipodal points of that circle which do not belong to the radial limit set, 
$(\xi\eta)$ consists of a  (non-compact) bunch of parallel geodesics (actually the whole flat submanifold as a set) and its boundary contains two radial limit points. Hence there exists an 
infinite number of elements $\gamma\in\Gamma$ such that the distance $d(\gamma\xo,(\xi\eta))$ is bounded. However, an equivalent phenomenon is not possible in dimension two as a classification of the ends of $M$ provides a complete description of the set $\joinrand$ of pairs of points 
which can be joined by a geodesic (see \cite{MR533654}  and   \cite{MR2255528} for a description of the geometry of ends in the universal covering useful for our purposes). So we get the following

\begin{lemma}\label{muradlim0dim2}\ 
If $\,\dim\XX=2\,$ and $\mu_\xo(\radlim)=0\,$ then $(\joinrand, \Gamma, (\mu_\xo\otimes\mu_\xo)\ein_{\joinrand})$ is completely dissipative.
\end{lemma}
\begin{proof} Assume that $(\xi,\eta)\in\joinrand\setminus\partial_1^2\XX$ is a pair of end points of a geodesic which bounds a flat half plane in $\XX$. Since $M$ is a non-elementary surface, $\xi$ and $\eta$ are the  attractive and repulsive fixed points of an axial isometry in $\Gamma$ which corresponds in the 
quotient to an isometry which is a rotation along a geodesic ray in a flat cylindral end of $M$. 
So both $\xi$ and $\eta$ belong to the radial limit set, and hence 
\[ \joinrand\subseteq \partial_1^2\XX \cup \bigl(\joinrand\cap (\radlim\times\radlim)\bigr).\]
From  $\mu_\xo(\radlim)=0\, $ we then obtain 
\[ (\mu_\xo\otimes\mu_\xo)\ein_{\joinrand}(\joinrand\setminus \partial_1^2\XX) \le (\mu_\xo\otimes\mu_\xo)\ein_{\joinrand}\bigl(\joinrand\cap (\radlim\times\radlim)\bigr)=0\]
and the claim follows from Lemma~\ref{muradlim0}. 
\end{proof}

In the case of surfaces we also have the following
\begin{lemma}\label{ergodicequiv2}
If $\ \dim \XX=2\, $ and $(\joinrand, \Gamma, (\mu_\xo\otimes\mu_\xo)\ein_{\joinrand})\,$ is ergodic, then the dynamical system $(\rand\times\rand, \Gamma, \mu_\xo\otimes\mu_\xo)\,$ is ergodic and completely conservative.
\end{lemma} 
\begin{proof} If $(\joinrand, \Gamma, (\mu_\xo\otimes\mu_\xo)\ein_{\joinrand})$ is ergodic, then $\mu_\xo(\radlim)=1$. From Proposition~\ref{atomicpart} we know that $\mu_\xo$ has no atoms, so   we get for the diagonal $D:=\{(\xi,\xi)\colon\xi\in\rand\}$ in $ \rand\times \rand$
\[ (\mu_\xo\otimes \mu_\xo)(D)=0.\]
Moreover, if $\xi,\eta\in\radlim$, then either $\xi=\eta$ (i.e. $\xi,\eta\in D$) or $(\xi\eta)$ is a rank one geodesic (i.e. $(\xi,\eta)\in\partial^2_1\XX$) or $(\xi\eta)$ is a flat half plane. In the last case, the only radial limit points  in the geometric boundary of $(\xi\eta)$ are $\xi$ and $\eta$. So we get
\[ \radlim\times\radlim\subseteq D \cup \partial_1^2\XX\cup \bigcup_{(\xi\eta)\in\joinrand\setminus  \partial^2_1\XX} \{\xi\}\times\{\eta\}\subseteq D \cup \partial^2\XX \]
and therefore
\begin{align*} 
(\mu_\xo\otimes \mu_\xo)(\rand\times\rand)&= (\mu_\xo\otimes \mu_\xo)(\radlim\times\radlim)\\ 
&\le (\mu_\xo\otimes \mu_\xo)(D)+(\mu_\xo\otimes\mu_\xo)(\partial^2\XX)=(\mu_\xo\otimes \mu_\xo)\ein_{\joinrand}(\joinrand).\end{align*}
\end{proof}

We remark that in higher dimension the existence and distribution of immersed flat submanifolds or higher rank symmetric spaces (of lower dimension than that of $M$) seems to be a difficult question (\cite{MR1132759}, \cite{MR908215},\cite{MR953675},\cite{MR994382}, \cite{MR1050413}); so we do not know whether Lemma~\ref{muradlim0dim2} and Lemma~\ref{ergodicequiv2} are true in dimension bigger than or equal to three. 

Together with Theorem~\ref{HTS} (from which we keep the notation), Lemma~\ref{ergodicequiv},~\ref{muradlim0dim2} and~\ref{ergodicequiv2} imply 

\begin{theorem}\label{HTS2}\ 
If $M=\XX/\Gamma$ is a non-elementary rank one surface, then, with the above notation, the conditions $(i)$ to $(iv)$ are equivalent in each of the two complementary cases:\\[2mm]
1.~Case:
\begin{enumerate}
\item[(i)] $\sum_{\gamma\in\Gamma} \e^{-\delta(\Gamma)d(\xo,\gamma\xo)}$ converges.
\item[(ii)] $\mu_\xo(\radlim)=0$.
\item[(iii)] $(\joinrand, \Gamma, (\mu_\xo\otimes\mu_\xo)\ein_{\joinrand})$ is completely dissipative and non-ergodic.
\item[(iv)] $(SM/\Gamma, (g_\Gamma^t)_{t\in\RR}, m_\Gamma)$ is completely dissipative and non-ergodic.
\end{enumerate}
2.~Case:
\begin{enumerate}
\item[(i)] $\sum_{\gamma\in\Gamma} \e^{-\delta(\Gamma)d(\xo,\gamma\xo)}$ diverges.
\item[(ii)] $\mu_\xo(\radlim)=1$.
\item[(iii)] $(\joinrand, \Gamma, \mu_\xo\otimes\mu_\xo)$ is completely conservative and ergodic.
\item[(iv)] $(SM, (g_\Gamma^t)_{t\in\RR}, m_\Gamma)$ is completely conservative and ergodic.
\end{enumerate}
\end{theorem}

\vspace{4mm}
\bibliography{BibHTS}

\medskip
Received xxxx 20xx; revised xxxx 20xx.
\medskip

\end{document}